\documentclass[10pt]{article}
\usepackage{amsmath,amscd}
\usepackage{amssymb,latexsym,amsthm}
\usepackage{color}
\usepackage[spanish,english]{babel}
\usepackage{amssymb}
\usepackage{color}
\usepackage{amsmath,amsthm,amscd}
\usepackage[latin1]{inputenc}
\usepackage{lscape}
\usepackage{fancyhdr}
\usepackage{amsfonts}
\usepackage{pb-diagram}

\numberwithin{equation}{section}
\newtheorem{theorem}{Theorem}[section]

\newtheorem{definition}[theorem]{Definition}
\newtheorem{proposition}[theorem]{Proposition}

\newtheorem{corollary}[theorem]{Corollary}

\theoremstyle{definition}

\newtheorem{example}[theorem]{Example}

\newtheorem{remark}[theorem]{Remark}

\newcommand{\cO}{\mbox{${\cal O}$}}

\newcommand{\cU}{\mbox{${\cal U}$}}

\newcommand{\cW}{\mbox{${\cal W}$}}

\title{\textbf{Koszulity of finitely semi-graded algebras}}
\author{José Oswaldo Lezama Serrano\\
\texttt{jolezamas@unal.edu.co}
\\Jaime Andrés Gómez Ortíz
\\ Seminario de Álgebra Constructiva - SAC$^2$\\ Departamento de Matemáticas\\ Universidad Nacional de
Colombia, Sede Bogot\'a}
\date{}
\begin{document}
\maketitle
\begin{abstract}
\noindent In this paper, we introduce the class of finitely semi-graded algebras which extends the
connected graded algebras finitely generated in degree one. The Koszul behavior of finitely
semi-graded algebras is investigated by the distributivity of some associated lattice of ideals.
The Hilbert series, the Poincaré series and the Yoneda algebra are defined for this class of
algebras. Finitely semi-graded algebras include many important examples of non $\mathbb{N}$-graded
algebras finitely generated in degree one coming from mathematical physics, and for these concrete
examples the Koszulity will be established, as well as, the explicit computation of its Hilbert and
Poincaré series.

\bigskip


\bigskip

\end{abstract}

\section{Introduction}

Finitely graded algebras over fields cover many important classes of non-commutative rings and
algebras coming from mathematical physics; examples of these algebras are the multi-parameter
quantum affine $n$-space, the Jordan plane, the Manin algebra $M_q(2)$, the multiplicative analogue
of the Weyl algebra, among many others. There exists recent interest in developing the
non-commutative projective algebraic geometry for finitely graded algebras (see for example
\cite{Ginzburg1}, \cite{Kanazawa}, \cite{Liu-Wu}, \cite{Liu-Wang}, \cite{Rogalski}, \cite{Suarez}).
However, for non $\mathbb{N}$-graded algebras only few works in this direction have been realized
(\cite{Gaddis}, \cite{lezamalatorre}). Some examples of non $\mathbb{N}$-graded algebras generated
in degree one are the dispin algebra $\cU(osp(1,2))$, the Woronowicz algebra
$\cW_{\nu}(\mathfrak{sl}(2,K))$, the quantum algebra $\cU'(\mathfrak{so}(3,K))$, the quantum
symplectic space $\cO_q(\mathfrak{sp}(K^{2n}))$, some algebra of operators, among others. One of
the most important algebraic properties studied in non-commutative algebraic geometry for graded
algebras is the Koszulity. Koszul graded algebras were defined by Stewart B. Priddy in
\cite{Priddy} and have many equivalent characterizations involving the Hilbert series, the Poincaré
series, the Yoneda algebra and some associated lattices of vector spaces. In this paper we are
interested in investigating the Koszul behavior for algebras over fields not being necessarily
$\mathbb{N}$-graded. For this purpose we will introduce in this work the finitely semi-graded
algebras; these type of algebras extend finitely graded algebras over fields generated in degree
one, and conform a particular subclass of finitely semi-graded rings defined in
\cite{lezamalatorre}. In order to study the Koszulity for finitely semi-graded algebras we will
define its Hilbert series, the Poincaré series, the Yoneda algebra, and we will investigate some
associated lattices of vector spaces similarly as this is done in the classical graded case.

For finitely semi-graded algebras we will prove the uniqueness of the Hilbert series (Corollary
\ref{corollary3.10}); in the proof we used a beautiful paper by Jason Bell and James J. Zhang
(\cite{ZhangJ4}), where this property was established for connected graded algebras finitely
generated in degree $1$. The uniqueness of the Poincaré series of a given finitely semi-graded
algebra was also proved assuming that its Yoneda algebra is finitely generated in degree one and
the base field has a free homogeneous resolution (Corollary \ref{corollary19.3.22}). We will see
that a finitely semi-graded algebra has a natural induced $\mathbb{N}$-filtration, so we will show
that the Hilbert series of the algebra coincides with the Hilbert series of its  associated graded
algebra. We will associate to a finitely semi-graded algebra a lattice of vector spaces defined
with the ideal of relations of its presentation, and from a result that gives conditions over the
distributiveness of this lattice (Theorem \ref{theorem19.4.4}), we will define the semi-graded
Koszul algebras, extending this way the classical notion of graded Koszul algebras. One important
part of the present paper consists in giving many examples of finitely semi-graded algebras as well
as examples of semi-graded Koszul algebras. Most of the examples that we will present arise in
mathematical physics and can be interpreted as skew $PBW$ extensions. This class of non-commutative
rings of polynomial type were introduced in \cite{Gallego2}, and they are a good global way of
describing rings and algebras not being necessarily $\mathbb{N}$-graded. Thus, the general results
that we will prove for finitely semi-graded algebras will be in particular applied to skew $PBW$
extensions; in Corollary \ref{corollary16.3.16} we explicitly computed the Hilbert series of skew
$PBW$ extensions that are finitely semi-graded algebras over fields, covering this way many
examples of algebras coming from quantum physics. Finally, in Theorem \ref{theorem19.4.13} and
Example \ref{example4.8} we present examples of non $\mathbb{N}$-graded algebras that have Koszul
behavior, i.e, they are semi-graded Koszul.

The paper is organized in the following way: In the first section we review the basic facts on
semi-graded rings and skew $PBW$ extensions that we need for the rest of the work. In the second
section we introduce the semi-graded algebras and we present many examples of them. The list of
examples include not only skew $PBW$ extensions that are algebras over fields, but also other non
graded algebras that can not be described as skew extensions. The third section is dedicated to
construct and prove the uniqueness of the Hilbert series, the Poincaré series and the Yoneda
algebra of a finitely semi-graded algebra. In the last section we study the Koszul behavior of
finitely semi-graded algebras and we will show that some non $\mathbb{N}$-graded algebras coming
from quantum physics are semi-graded Koszul.

If not otherwise noted, all modules are left modules and $K$ will be an arbitrary field.

In order to appreciate better the results of the paper we recall first the definition of finitely
graded algebras over fields and its Hilbert series (see \cite{Rogalski}). Let $A$ be a $K$-algebra,
$A$ is \textit{finitely graded} if: (a) $A$ is $\mathbb{N}$-graded, i.e., $A$ has a graduation
$A=\bigoplus_{n\geq 0} A_n$, $A_nA_m\subseteq A_{n+m}$ for every $n,m\geq 0$; (b) $A$ is
\textit{connected}, i.e., $A_0=K$; (c) $A$ is finitely generated as $K$-algebra. Thus, $A$ is
\textit{locally finite}, i.e.,  $\dim_K A_n<\infty$ for every $n\geq 0$, and hence the
\textit{Hilbert series} of $A$ is defined by
\begin{center}
$h_A(t):=\sum_{n=0}^\infty(\dim_K A_n)t^n$.
\end{center}

\subsection{Semi-graded rings and modules}

\noindent In this starting subsection we recall the definition and some basic facts about
semi-graded rings and modules, more details and the proofs omitted here can be found in
\cite{lezamalatorre}.

\begin{definition}
Let $B$ be a ring. We say that $B$ is semi-graded $(SG)$ if there exists a collection
$\{B_n\}_{n\geq 0}$ of subgroups $B_n$ of the additive group $B^+$ such that the following
conditions hold:
\begin{enumerate}
\item[\rm (i)]$B=\bigoplus_{n\geq 0}B_n$.
\item[\rm (ii)]For every $m,n\geq 0$, $B_mB_n\subseteq B_0\oplus \cdots \oplus B_{m+n}$.
\item[\rm (iii)]$1\in B_0$.
\end{enumerate}
The collection $\{B_n\}_{n\geq 0}$ is called a semi-graduation of $B$ and we say that the elements
of $B_n$ are homogeneous of degree $n$. Let $B$ and $C$ be semi-graded rings and let $f: B\to C$ be
a ring homomorphism, we say that $f$ is homogeneous if $f(B_n)\subseteq C_{n}$ for every $n\geq 0$.
\end{definition}

\begin{definition}
Let $B$ be a $SG$ ring and let $M$ be a $B$-module. We say that $M$ is a $\mathbb{Z}$-semi-graded,
or simply semi-graded, if there exists a collection $\{M_n\}_{n\in \mathbb{Z}}$ of subgroups $M_n$
of the additive group $M^+$ such that the following conditions hold:
\begin{enumerate}
\item[\rm (i)]$M=\bigoplus_{n\in \mathbb{Z}} M_n$.
\item[\rm (ii)]For every $m\geq 0$ and $n\in \mathbb{Z}$, $B_mM_n\subseteq \bigoplus_{k\leq m+n}M_k$.
\end{enumerate}
The collection $\{M_n\}_{n\in \mathbb{Z}}$ is called a \textit{semi-graduation} of $M$ and we say
that the elements of $M_n$ are \textit{homogeneous} of degree $n$. We say that $M$ is positively
semi-graded, also called $\mathbb{N}$-semi-graded, if $M_n=0$ for every $n<0$. Let $f: M\to N$ be
an homomorphism of $B$-modules, where $M$ and $N$ are semi-graded $B$-modules; we say that $f$ is
homogeneous if $f(M_n)\subseteq N_n$ for every $n\in \mathbb{Z}$.
\end{definition}

Let $B$ be a semi-graded ring and $M$ be a semi-graded $B$-module, let $N$ be a submodule of $M$
and $N_n:=N\cap M_n$, $n\in \mathbb{Z}$; observe that the sum $\sum_{n}N_n$ is direct. This induces
the following definition.

\begin{definition}
Let $B$ be a $SG$ ring and $M$ be a semi-graded module over $B$. Let $N$ be a submodule of $M$, we
say that $N$ is a semi-graded submodule of $M$ if $N=\bigoplus_{n\in \mathbb{Z}}N_n$.
\end{definition}

We present next an important class of semi-graded rings that includes finitely graded algebras.

\begin{definition}\label{definition17.5.4}
Let $B$ be a ring. We say that $B$ is finitely semi-graded $(FSG)$ if $B$ satisfies the following
conditions:
\begin{enumerate}
\item[\rm (i)]$B$ is $SG$.
\item[\rm (ii)]There exists finitely many elements $x_1,\dots,x_n\in B$ such that the
subring generated by $B_0$ and $x_1,\dots,x_n$ coincides with $B$.
\item[\rm (iii)]For every $n\geq 0$, $B_n$ is a free $B_0$-module of finite dimension.
\end{enumerate}
Moreover, if $M$ is a $B$-module, we say that $M$ is finitely semi-graded if $M$ is semi-graded,
finitely generated, and for every $n\in \mathbb{Z}$, $M_n$ is a free $B_0$-module of finite
dimension.
\end{definition}

From the definitions above we get the following elementary but key facts.

\begin{proposition}\label{proposition17.5.5}
Let $B=\bigoplus_{n\geq 0}B_n$ be a $SG$ ring. Then,
\begin{enumerate}
\item[\rm (i)]$B_0$ is a subring of $B$. Moreover, for any $n\geq 0$, $B_0\oplus \cdots \oplus B_{n}$ is a $B_0-B_0$-bimodule, as well as $B$.
\item[\rm (ii)]$B$ has a standard $\mathbb{N}$-filtration given by
\begin{equation}\label{equ17.5.1}
F_n(B):=B_0\oplus \cdots \oplus B_{n}.
\end{equation}
\item[\rm (iii)]The associated graded ring $Gr(B)$ satisfies
\begin{center}
$Gr(B)_n\cong B_n$, for every $n\geq 0$ $($isomorphism of abelian groups$)$.
\end{center}
\item[\rm (iv)]Let $M=\bigoplus_{n\in \mathbb{Z}}M_n$ be a semi-graded $B$-module and $N$ a submodule of $M$. The following conditions are equivalent:
\begin{enumerate}
\item[\rm (a)]$N$ is semi-graded.
\item[\rm (b)]For every $z\in N$, the homogeneous components of $z$ are in $N$.
\item[\rm (c)]$M/N$ is semi-graded with semi-graduation given by
\begin{center}
$(M/N)_n:=(M_n+N)/N$, $n\in \mathbb{Z}$.
\end{center}
\end{enumerate}
\end{enumerate}
\end{proposition}

\begin{remark}\label{remark18.1.7}
(i) According to (iv)-(b) in the previous proposition, if $N$ is a semi-graded submodule of $M$,
then $N$ can be generated by homogeneous elements; however, if $N$ is a submodule of $M$ generated
by homogeneous elements, then we can not asserts that $N$ is semi-graded.

(ii) Let $B$ be a $SG$ ring, as we saw in (ii) of the previous proposition, $B$ is
$\mathbb{N}$-filtered. Conversely, if we assume that $B$ is a $\mathbb{N}$-filtered ring with
filtration $\{F_n(B)\}_{n\geq 0}$ such that for any $n\geq 0$, $F_n(B)/F_{n-1}(B)$ is
$F_0(B)$-projective, then it is easy to prove that $B$ is $SG$ with semi-graduation $\{B_n\}_{n\geq
0}$ given by $B_0:=F_0(B)$ and $B_n$ is such that $F_{n-1}(B)\oplus B_n=F_n(B)$, $n\geq 1$.

(iii) If $B$ is a $FSG$ ring, then for every $n\geq 0$, $Gr(B)_n\cong B_n$ as $B_0$-modules.

(iv) Observe if $B$ is $FSG$ ring, then $B_0B_p=B_p$ for every $p\geq 0$, and if $M$ is finitely
semi-graded, then $B_0M_n=M_n$ for all $n\in \mathbb{Z}$.
\end{remark}

We conclude this subsection recalling one of the invariants that we will study later for finitely
semi-graded algebras. In \cite{lezamalatorre} the authors introduced the notion of generalized
Hilbert series for finitely semi-graded rings.

\begin{definition}\label{definition15.2.1}
Let $B=\bigoplus_{n\geq 0}\oplus B_n$ be a $FSG$ ring. The generalized Hilbert series of $B$ is
defined by
\begin{center}
$Gh_B(t):=\sum_{n=0}^\infty(\dim_{B_0} B_n)t^n$.
\end{center}
\end{definition}

\begin{remark}\label{remark17.5.11}
(i) Note that if $K$ is a field and $B$ is a finitely graded $K$-algebra, then the generalized
Hilbert series coincides with the usual Hilbert series, i.e., $Gh_{B}(t)=h_B(t)$.

(ii) Observe that if a $FSG$ ring $B$ has another \-se\-mi-\-gra\-dua\-tion $B=\bigoplus_{n\geq 0}
C_n$, then its generalized Hilbert series can change, i.e., the notion of generalized Hilbert
series depends on the semi-graduation, in particular on $B_0$. For example, consider the usual real
polynomial ring in two variables $B:=\mathbb{R}[x,y]$, then $Gh_{B}(t)=\frac{1}{(1-t)^2}$, but if
we view this ring as $B=(\mathbb{R}[x])[y]$ then $C_0=\mathbb{R}[x]$, its generalized Hilbert
series is $\frac{1}{1-t}$. However, in Section \ref{section3} we will introduce the semi-graded
algebras over fields and for them we will discuss the uniqueness of the Hilbert series based in a
recent paper by Bell and Zhang (\cite{ZhangJ4}).
\end{remark}

\subsection{Skew $PBW$ extensions}

As was pointed out above, finitely graded algebras over fields are examples of $FSG$ rings. In
order to present many other examples of $FSG$ rings not being necessarily graded algebras, we
recall in this subsection the notion of skew $PBW$ extension defined firstly in \cite{Gallego2}.

\begin{definition}[\cite{Gallego2}]\label{gpbwextension}
Let $R$ and $A$ be rings. We say that $A$ is a \textit{skew $PBW$ extension of $R$} $($also called
a $\sigma-PBW$ extension of $R$$)$ if the following conditions hold:
\begin{enumerate}
\item[\rm (i)]$R\subseteq A$.
\item[\rm (ii)]There exist finitely many elements $x_1,\dots ,x_n\in A$ such $A$ is a left $R$-free module with basis
\begin{center}
${\rm Mon}(A):= \{x^{\alpha}=x_1^{\alpha_1}\cdots x_n^{\alpha_n}\mid \alpha=(\alpha_1,\dots
,\alpha_n)\in \mathbb{N}^n\}$, with $\mathbb{N}:=\{0,1,2,\dots\}$.
\end{center}
The set ${\rm Mon}(A)$ is called the set of standard monomials of $A$.
\item[\rm (iii)]For every $1\leq i\leq n$ and $r\in R-\{0\}$ there exists $c_{i,r}\in R-\{0\}$ such that
\begin{equation}\label{sigmadefinicion1}
x_ir-c_{i,r}x_i\in R.
\end{equation}
\item[\rm (iv)]For every $1\leq i,j\leq n$ there exists $c_{i,j}\in R-\{0\}$ such that
\begin{equation}\label{sigmadefinicion2}
x_jx_i-c_{i,j}x_ix_j\in R+Rx_1+\cdots +Rx_n.
\end{equation}
\end{enumerate}
Under these conditions we will write $A:=\sigma(R)\langle x_1,\dots ,x_n\rangle$.
\end{definition}

\begin{example}\label{example1.10}
Many important algebras and rings coming from mathematical physics are particular examples of skew
$PBW$ extensions: Habitual ring of polynomials in several variables, Weyl algebras, enveloping
algebras of finite dimensional Lie algebras, algebra of $q$-differential operators, many important
types of Ore algebras, algebras of diffusion type, additive and multiplicative analogues of the
Weyl algebra, dispin algebra $\mathcal{U}(osp(1,2))$, quantum algebra $\mathcal{U}'(so(3,K))$,
Woronowicz algebra $\mathcal{W}_{\nu}(\mathfrak{sl}(2,K))$, Manin algebra $\mathcal{O}_q(M_2(K))$,
coordinate algebra of the quantum group $SL_q(2)$, $q$-Heisenberg algebra \textbf{H}$_n(q)$,
Hayashi algebra $W_q(J)$, differential operators on a quantum space
$D_{\textbf{q}}(S_{\textbf{q}})$, Witten's deformation of $\mathcal{U}(\mathfrak{sl}(2,K))$,
multiparameter Weyl algebra $A_n^{Q,\Gamma}(K)$, quantum symplectic space
$\mathcal{O}_q(\mathfrak{sp}(K^{2n}))$, some quadratic algebras in 3 variables, some 3-dimensional
skew polynomial algebras, particular types of Sklyanin algebras, homogenized enveloping algebra
$\mathcal{A}(\mathcal{G})$, Sridharan enveloping algebra of 3-dimensional Lie algebra
$\mathcal{G}$, among many others. For a precise definition of any of these rings and algebras see
\cite{lezamareyes1}, \cite{Reyes2}, \cite{Suarez2}, \cite{Suarez}, \cite{ReyesSuarez2}.
\end{example}

Associated to a skew $PBW$ extension $A=\sigma(R)\langle x_1,\dots ,x_n\rangle$, there are $n$
injective endomorphisms $\sigma_1,\dots,\sigma_n$ of $R$ and $\sigma_i$-derivations, as the
following proposition shows.
\begin{proposition}[\cite{Gallego2}]\label{sigmadefinition}
Let $A$ be a skew $PBW$ extension of $R$. Then, for every $1\leq i\leq n$, there exist an injective
ring endomorphism $\sigma_i:R\rightarrow R$ and a $\sigma_i$-derivation $\delta_i:R\rightarrow R$
such that
\begin{center}
$x_ir=\sigma_i(r)x_i+\delta_i(r)$,
\end{center}
for each $r\in R$.
\end{proposition}
A particular case of skew $PBW$ extension is when all derivations $\delta_i$ are zero. Another
interesting case is when all $\sigma_i$ are bijective and the constants $c_{ij}$ are invertible. We
recall the following definition.
\begin{definition}[\cite{Gallego2}, \cite{Suarez2}, \cite{Suarez}, \cite{ReyesSuarez2}]\label{sigmapbwderivationtype}
Let $A$ be a skew $PBW$ extension.
\begin{enumerate}
\item[\rm (a)]
$A$ is quasi-commutative if the conditions {\rm(}iii{\rm)} and {\rm(}iv{\rm)} in Definition
\ref{gpbwextension} are replaced by
\begin{enumerate}
\item[\rm (iii')]For every $1\leq i\leq n$ and $r\in R-\{0\}$ there exists $c_{i,r}\in R-\{0\}$ such that
\begin{equation}
x_ir=c_{i,r}x_i.
\end{equation}
\item[\rm (iv')]For every $1\leq i,j\leq n$ there exists $c_{i,j}\in R-\{0\}$ such that
\begin{equation}
x_jx_i=c_{i,j}x_ix_j.
\end{equation}
\end{enumerate}
\item[\rm (b)]$A$ is bijective if $\sigma_i$ is bijective for
every $1\leq i\leq n$ and $c_{i,j}$ is invertible for any $1\leq i,j\leq n$.
\item[\rm (c)]$A$ is constant if the condition  {\rm(}ii{\rm)} in Definition
\ref{gpbwextension} is replaced by: For every $1\leq i\leq n$ and $r\in R$,
\begin{equation}
x_ir=rx_i.
\end{equation}
\item[\rm (d)]$A$ is pre-commutative if the condition  {\rm(}iv{\rm)} in Definition
\ref{gpbwextension} is replaced by: For any $1\leq i,j\leq n$ there exists $c_{i,j}\in R\
\backslash\ \{0\}$ such that
\begin{equation}\label{relat.pre-comm}
x_jx_i-c_{i,j}x_ix_j\in Rx_1+\cdots +Rx_n.
\end{equation}
\item[\rm (e)]  $A$ is called {\em semi-commutative} if $A$ is quasi-commutative and constant.
\end{enumerate}
\end{definition}

\begin{remark}\label{remark1.13}
Later below we need the following classification given in \cite{Suarez2}, \cite{Suarez} and
\cite{ReyesSuarez2} of skew $PBW$ extensions of Example \ref{example1.10}. The extensions are
classified as constant (C), bijective (B), pre-commutative (P), quasi-commutative (QC) and
\-semi-\-commu\-tative (SC); in the tables the symbols $\star$ and $\checkmark$ denote negation and
affirmation, respectively:

\newpage

\begin{center}

\scriptsize{
\begin{tabular}{|l|c|c|c|c|c|}\hline
 \textbf{Skew PBW extension} &   \textbf{C} & \textbf{B} & \textbf{P} & \textbf{QC} &  \textbf{SC} \\
\hline \hline  Classical polynomial ring   & $\checkmark$ & $\checkmark$ & $\checkmark$  &
$\checkmark$ &  $\checkmark$\\ \hline
 Ore extensions  of bijective type  & $\star$ & $\checkmark$ &  $\checkmark$ & $\star$ & $\star$\\ \hline
 Weyl algebra  &  $\star$ & $\checkmark$  &  $\checkmark$ & $\star$ & $\star$ \\ \hline
 Particular  Sklyanin algebra &  $\checkmark$ & $\checkmark$ & $\checkmark$ &  $\checkmark$ &  $\checkmark$ \\ \hline
 Universal enveloping algebra of a Lie algebra  & $\checkmark$ & $\checkmark$ &   $\checkmark$ & $\star$ & $\star$\\ \hline
 Homogenized enveloping algebra $\mathcal{A}(\mathcal{G})$ & $\checkmark$ & $\checkmark$ &   $\checkmark$ & $\star$ & $\star$\\ \hline
Tensor product  & $\checkmark$ &$\checkmark$ & $\checkmark$& $\star$ &$\star$ \\ \hline
 Crossed product  & $\star$ & $\checkmark$ & $\star$ &$\star$ & $\star$\\ \hline
 Algebra of $q$-differential operators   & $\star$ & $\checkmark$ & $\checkmark$ &$\star$ & $\star$\\ \hline
 Algebra of shift operators   & $\star$ & $\checkmark$ &$\checkmark$ &$\checkmark$ & $\star$ \\ \hline
 Mixed algebra  & $\star$& $\checkmark$ &$\star$ &$\star$ &$\star$ \\ \hline
 Algebra of discrete linear systems   & $\star$&$\checkmark$ &$\checkmark$ &$\checkmark$ & $\star$\\ \hline
 Linear partial differential operators   & $\star$& $\checkmark$& $\checkmark$&$\star$ &$\star$ \\ \hline
 Linear partial shift operators   &$\star$ &$\checkmark$ &$\checkmark$ & $\checkmark$&$\star$ \\ \hline
 Algebra of linear partial difference operators   & $\star$& $\checkmark$& $\checkmark$ &$\star$ & $\star$\\ \hline
 Algebra of linear partial q-dilation operators   & $\star$ &$\checkmark$ & $\checkmark$&$\checkmark$ &$\star$ \\ \hline
 Algebra of linear partial q-differential operators   &$\star$ & $\checkmark$ & $\checkmark$ &$\star$ &$\star$ \\ \hline
 Algebras of diffusion type  & $\checkmark$ & $\checkmark$ & $\checkmark$ & $\star$ & $\star$\\ \hline
  Additive analogue of the Weyl algebra   &$\checkmark$ &$\checkmark$ & $\star$& $\star$& $\star$\\ \hline
 Multiplicative analogue of the Weyl algebra  &$\checkmark$ &$\checkmark$ &$\checkmark$ & $\checkmark$& $\checkmark$\\ \hline
 Quantum algebra $\cU'(\mathfrak{so}(3,K))$  & $\checkmark$ &$\checkmark$ &$\checkmark$ & $\star$& $\star$\\ \hline
 Dispin algebra  & $\checkmark$ & $\checkmark$ & $\checkmark$ & $\star$ & $\star$\\ \hline
 Woronowicz algebra  &$\checkmark$ &$\checkmark$ &$\checkmark$ & $\star$& $\star$ \\ \hline
 Complex algebra    & $\star$  &$\checkmark$ &$\star$  &$\star$  & $\star$ \\ \hline
 Algebra $\textbf{U}$   & $\star$  &$\checkmark$ &$\star$  &$\star$  & $\star$ \\ \hline
Manin algebra   & $\star$  &$\checkmark$ & $\checkmark$  &$\star$  & $\star$ \\ \hline
$q$-Heisenberg algebra  & $\checkmark$  &$\checkmark$ & $\checkmark$  &$\star$  & $\star$  \\
\hline
 Quantum enveloping algebra of $\mathfrak{sl}(2,\mathbb{K})$ & $\star$  &$\checkmark$ &$\star$  &$\star$  & $\star$ \\ \hline
 Hayashi's algebra  & $\star$  &$\checkmark$ &$\star$  &$\star$  & $\star$ \\ \hline
  The algebra of differential operators on a quantum space $S_q$ & $\star$ & $\checkmark$  & $\star$ & $\star$ &$\star$ \\ \hline
Witten's deformation of   $\cU(\mathfrak{sl}(2,\mathbb{K}))$& $\star$ & $\checkmark$  & $\star$ &
$\star$ &$\star$  \\ \hline
 Quantum Weyl algebra of Maltsiniotis  & $\star$ & $\checkmark$  & $\star$ & $\star$ &$\star$ \\ \hline
 Quantum Weyl algebra  & $\star$ & $\checkmark$  & $\star$ & $\star$ &$\star$  \\ \hline
 Multiparameter quantized Weyl algebra & $\star$ & $\checkmark$  & $\star$ & $\star$ &$\star$  \\ \hline
 Quantum symplectic space & $\star$ & $\checkmark$  & $\star$ & $\star$ &$\star$ \\ \hline
 Quadratic algebras in 3 variables & $\star$ & $\checkmark$  & $\star$ & $\star$ &$\star$  \\ \hline
\end{tabular}}
\end{center}

\begin{center}
\tiny{
\begin{tabular}{|c|p{8.7cm}|c|c|c|c|c|}\hline
\textbf{Cardinal} & \textbf{ 3-dimensional skew polynomial algebras} &
 \textbf{C} & \textbf{B} & \textbf{P} &  \textbf{QC}&  \textbf{SC} \\ \hline \hline
$|\{\alpha, \beta, \gamma\}|=3$ & $yz-\alpha zy=0, \ \ zx-\beta xz=0, \ \ xy-\gamma yx=0$ &
$\checkmark$  & $\checkmark$  & $\checkmark$  & $\checkmark$  & $\checkmark$  \\ \cline{1-7} &
$yz-zy=z$,\ \ $zx-\beta xz=y$,\ \ $xy-yx=x$ & $\checkmark$  & $\checkmark$  & $\checkmark$  &
$\star$  & $\star$ \\ \cline{2-7} & $yz-zy=z,\ \ zx-\beta xz=b,\ \ xy-yx=x $ & $\checkmark$ &
$\checkmark$ & $\star$  & $\star$   & $\star$ \\ \cline{2-7} $|\{\alpha, \beta, \gamma\}|=2$, &
$yz-zy=0,\ \ zx-\beta xz=y,\ \ xy-yx=0$ & $\checkmark$  & $\checkmark$  & $\checkmark$  & $\star$ &
$\star$\\ \cline{2-7} $\beta\neq \alpha=\gamma=1$ & $yz-zy=0,\ \ zx-\beta xz=b,\ \ xy-yx=0$ &
$\checkmark$ & $\checkmark$ & $\star$  & $\star$   & $\star$ \\\cline{2-7} & $yz-zy=az,\ \ zx-\beta
xz=0,\ \ xy-yx=x$ & $\checkmark$  & $\checkmark$  & $\checkmark$  & $\star$  & $\star$\\
\cline{2-7} & $ yz-zy=z,\ \ zx-\beta xz=0,\ \ xy-yx=0$ & $\checkmark$  & $\checkmark$  &
$\checkmark$  & $\star$  & $\star$\\ \cline{1-7} $|\{\alpha, \beta, \gamma\}|=2,$ & $yz-\alpha
zy=0,\ \ zx-\beta xz=y+b,\ \ xy-\alpha yx=0$ & $\checkmark$ & $\checkmark$ & $\star$  & $\star$   &
$\star$ \\ \cline{2-7} $\beta \neq \alpha = \gamma \neq 1$ & $yz-\alpha zy=0,\ \ zx-\beta xz=b,\ \
xy-\alpha yx=0$ & $\checkmark$ & $\checkmark$ & $\star$  & $\star$   & $\star$ \\ \cline{1-7}
$\alpha =\beta = \gamma \neq 1$ & $yz-\alpha zy=a_1x+b_1,\ \ zx-\alpha xz=a_2y + b_2,\ \ xy-\alpha
yx=a_3z+b_3 $ & $\checkmark$ & $\checkmark$ & $\star$  & $\star$   & $\star$ \\ \cline{1-7} &
$yz-zy=x,\ \ zx-xz=y,\ \ xy-yx =z $ & $\checkmark$  & $\checkmark$  & $\checkmark$  & $\star$  &
$\star$\\ \cline{2-7} & $yz-zy=0,\ \ zx-xz=0,\ \ xy-yx=z$ & $\checkmark$  & $\checkmark$  &
$\checkmark$  & $\star$  & $\star$ \\ \cline{2-7} $\alpha =\beta = \gamma =1$ & $yz-zy=0,\ \
zx-xz=0,\ \ xy-yx=b$ & $\checkmark$ & $\checkmark$ & $\star$  & $\star$   & $\star$\\ \cline{2-7} &
$yz-zy=-y,\ \ zx-xz=x+y,\ \ xy-yx=0$& $\checkmark$  & $\checkmark$  & $\checkmark$  & $\star$  &
$\star$ \\ \cline{2-7} & $ yz-zy=az,\ \ zx-xz=x,\ \ xy-yx=0 $ & $\checkmark$  & $\checkmark$  &
$\checkmark$  & $\star$  & $\star$\\ \cline{1-7}
\end{tabular}}
\end{center}

\begin{center}
\scriptsize{
\begin{tabular}{|c|c|c|c|c|c|c|c|c|}\hline
\multicolumn{9}{|c|}{\textbf{ Sridharan enveloping algebra of 3-dimensional Lie algebra
$\mathcal{G}$}}\\ \hline \textbf{Type} & $[x,y]$ & $[y,z]$ & $[z,x]$ & \textbf{C} & \textbf{B} &
\textbf{P} & \textbf{QC} &  \textbf{SC}\\ \hline\hline

 1 & $0$ & $0$ & $0$ & $\checkmark$  & $\checkmark$  & $\checkmark$  & $\checkmark$  & $\checkmark$   \\ \hline
 2 & $0$ & $x$ & $0$ & $\checkmark$ & $\checkmark$ & $\checkmark$  & $\star$  & $\star$ \\ \hline
 3& $x$ & $0$ & $0$ & $\checkmark$ & $\checkmark$ & $\checkmark$  & $\star$  & $\star$ \\ \hline
4 & $0$ & $\alpha y$ & $-x$ & $\checkmark$ & $\checkmark$ & $\checkmark$  & $\star$  & $\star$ \\
\hline 5 & $0$ & $y$ & $-(x+y)$ & $\checkmark$ & $\checkmark$ & $\checkmark$  & $\star$  & $\star$
\\ \hline
 6 & $z$ & $-2y$ & $-2x$ & $\checkmark$ & $\checkmark$ & $\checkmark$  & $\star$  & $\star$ \\ \hline
 7& $1$ & $0$ & $0$ & $\checkmark$ & $\checkmark$ & $\star$   & $\star$  & $\star$ \\ \hline
  8& $1$ & $x$ & $0$ & $\checkmark$ & $\checkmark$ & $\star$   & $\star$  & $\star$ \\ \hline
  9& $x$ & $1$ & $0$ &  $\checkmark$ & $\checkmark$ & $\star$   & $\star$  & $\star$ \\ \hline
   10& $1$ & $y$ & $x$ &  $\checkmark$ & $\checkmark$ & $\star$   & $\star$  & $\star$ \\ \hline
\end{tabular}}
\end{center}
\end{remark}

If $A=\sigma(R)\langle x_1,\dots,x_n\rangle$ is a skew $PBW$ extension of the ring $R$, then, as
was observed in Proposition \ref{sigmadefinition}, $A$ induces injective endomorphisms
$\sigma_k:R\to R$ and $\sigma_k$-derivations $\delta_k:R\to R$, $1\leq k\leq n$. From the
Definition \ref{gpbwextension}, there exists a unique finite set of constants $c_{ij}, d_{ij},
a_{ij}^{(k)}\in R$, $c_{ij}\neq 0$, such that
\begin{equation}\label{equation1.2.1}
x_jx_i=c_{ij}x_ix_j+a_{ij}^{(1)}x_1+\cdots+a_{ij}^{(n)}x_n+d_{ij}, \ \text{for every}\  1\leq
i<j\leq n.
\end{equation}

\begin{definition}\label{definition1.2.1}
Let $A=\sigma(R)\langle x_1,\dots,x_n\rangle$ be a skew $PBW$ extension. $R$, $n$,
$\sigma_k,\delta_k, c_{ij}$, $d_{ij}, a_{ij}^{(k)}$, with $1\leq i<j\leq n$, $1\leq k\leq n$,
defined as before, are called the parameters of $A$.
\end{definition}

Some notation will be useful in what follows.

\begin{definition}\label{1.1.6}
Let $A$ be a skew $PBW$ extension of $R$.
\begin{enumerate}
\item[\rm (i)]For $\alpha=(\alpha_1,\dots,\alpha_n)\in \mathbb{N}^n$,
$|\alpha|:=\alpha_1+\cdots+\alpha_n$.
\item[\rm (ii)]For $X=x^{\alpha}\in Mon(A)$,
$\exp(X):=\alpha$ and $\deg(X):=|\alpha|$.
\item[\rm (iii)]Let $0\neq f\in A$, $t(f)$ is the finite
set of terms that conform $f$, i.e., if $f=c_1X_1+\cdots +c_tX_t$, with $X_i\in Mon(A)$ and $c_i\in
R-\{0\}$, then $t(f):=\{c_1X_1,\dots,c_tX_t\}$.
\item[\rm (iv)]Let $f$ be as in {\rm(iii)}, then $\deg(f):=\max\{\deg(X_i)\}_{i=1}^t.$
\end{enumerate}
\end{definition}

Skew $PBW$ extensions have been enough investigated, many ring and homological properties of them
have been studied, as well as their Gröbner theory (\cite{Acosta1}, \cite{Acosta2},
\cite{lezamaore}, \cite{ArtamonovDerivations}, \cite{Gallego2}, \cite{lezama-gallego-projective},
\cite{lezamareyes1}, \cite{lezamalatorre}, \cite{LezamaHelbert}, \cite{Reyes2}, \cite{ReyesSuarez},
\cite{reyessuarez}, \cite{Venegas2}). We conclude this introductory section with some known results
about skew $PBW$ extensions and semi-graded rings that we will use in the present paper.

\begin{theorem}[\cite{lezamareyes1}]\label{filteredskew}
    Let $ A $ be an arbitrary skew $PBW$ extension of the ring $ R $. Then, $ A $ is a $\mathbb{N}$-filtered ring with filtration given by
    $$F_{m}:=
    \begin{cases}
    R, & \text{ if } m=0\\
    \{f \in A| deg(f) \leq m\}, & \text{  if } m \geq 1,
    \end{cases}$$
    and the graded ring $ Gr(A) $ is a quasi-commutative skew $ PBW $ extension of $ R
    $. If the parameters that define $A$ are as in Definition \ref{definition1.2.1}, then the
    parameters that define $Gr(A)$ are $R$, $n$, $\sigma_k, c_{ij}$, with $1\leq i<j\leq n$, $1\leq k\leq
    n$. Moreover, if $ A $ is bijective, then $ Gr(A) $ is bijective.
     \end{theorem}

\begin{proposition}[\cite{lezamalatorre}]\label{proposition16.5.7}
{\rm (i)} Any $\mathbb{N}$-graded ring is $SG$.

{\rm (ii)} Let $K$ be a field. Any finitely graded $K$-algebra is a $FSG$ ring.

{\rm (iii)} Any skew $PBW$ extension is a $FSG$ ring.
\end{proposition}

For skew $PBW$ extensions the generalized Hilbert series has been computed explicitly.

\begin{theorem}[\cite{lezamalatorre}]\label{17.5.10}
Let $A=\sigma(R)\langle x_1,\dots,x_n\rangle$ be an arbitrary skew $PBW$ extension. Then,
\begin{equation}\label{equ17.2.3}
Gh_A(t)=\frac{1}{(1-t)^n}.
\end{equation}
\end{theorem}

\begin{remark}\label{remark1.19}
(i) Note that the class of $SG$ rings includes properly the class of $\mathbb{N}$-graded rings: In
fact, the enveloping algebra of any finite-dimensional Lie algebra proves this statement. This
example proves also that the class of $FSG$ rings includes properly the class of finitely graded
algebras.

(ii) The class of $FSG$ rings includes properly the class of skew $PBW$ extensions: For this
consider the Artin-Schelter regular algebra of global dimension $3$ defined by the following
relations:
\begin{center}
$yx=xy+z^2$, $zy=yz+x^2$, $zx=xz+y^2$.
\end{center}
Observe that this algebra is a particular case of a Sklyanin algebra which in general are defined
by the following relations:
\begin{center}
$ayx+bxy+cz^2=0$, $azy+byz+cx^2=0$, $axz+bzx+cy^2=0$, $a,b,c\in K$.
\end{center}
\end{remark}

\section{Finitely semi-graded algebras}

\noindent In the present section we define the finitely semi-graded algebras. In all of examples
that we will study, in particular, the semi-graded Koszul algebras that we will introduce later,
they are additionally finitely presented. Let us recall first this notion. Let $B$ be a finitely
generated $K$-algebra, so there exist finitely many elements $g_1,\dots,g_n\in B$ that generate $B$
as $K$-algebra and we have the $K$-algebra homomorphism $f:K\{ x_1,\dots,x_n\}\to B$, with
$f(x_i):=g_i$, $1\leq i\leq n$; let $I:=\ker(f)$, then we get a \textit{presentation} of $B$:
\begin{equation}\label{equ17.1.1}
B\cong K\{ x_1,\dots,x_n\}/I.
\end{equation}
Recall that $B$ is said to be \textit{finitely presented} if $I$ is finitely generated.

\subsection{Definition}

In the previous section we defined the finitely semi-graded rings and we observed that they
generalize finitely graded algebras over fields and skew $PBW$ extensions. In this section we will
be concentrated in some particular class of this type of rings which satisfy some other extra
natural conditions.

\begin{definition}\label{definition16.1.1}
Let $B$ be a $K$-algebra. We say that $B$ is finitely semi-graded $(FSG)$ if the following
conditions hold:
\begin{enumerate}
\item[\rm (i)]$B$ is a $FSG$ ring with semi-graduation $B=\bigoplus_{p\geq 0}B_p$.
\item[\rm (ii)]For every $p,q\geq 1$, $B_pB_q\subseteq B_1\oplus \cdots \oplus B_{p+q}$.
\item[\rm (iii)]$B$ is connected, i.e., $B_0=K$.
\item[\rm (iv)]$B$ is generated in degree $1$.
\end{enumerate}
\end{definition}

\begin{remark}\label{17.1.2}
Let $B$ be a $FSG$ $K$-algebra;

(i) Since $B$ is locally finite and $B$ is finitely generated in degree $1$, then any $K$-basis of
$B_1$ generates $B$ as $K$-algebra.

(ii) The canonical projection $\varepsilon:B\to K$ is a homomorphism of $K$-algebras, called the
\textit{augmentation map}, with $\ker(\varepsilon)=\bigoplus_{n\geq 1}B_n$. Therefore, the class of
$FSG$ algebras is contained in the class of \textit{augmented algebras}, i.e., algebras with
augmentation (see \cite{Phan}), however, as we will see, a semi-graduation is a nice tool for
defining some invariants useful for the study of the algebra. $B_{\geq 1}:=\bigoplus_{n\geq 1}B_n$
is called the \textit{augmentation ideal}. Thus, $K$ becomes into a $B$-bimodule with products
given by $b\cdot \lambda:=b_0\lambda$, $\lambda\cdot b:=\lambda b_0$, with $b\in B, \lambda\in K$
and $b_0$ is the homogeneous component of $b$ of degree zero.

(iii) It is well known that $B$ is finitely graded if and only if the ideal $I$ in
(\ref{equ17.1.1}) is homogeneous (\cite{Rogalski}). In general, finitely semi-graded algebras do
not need to be finitely presented. Any finitely graded algebra generated in degree $1$ is $FSG$,
but $B:=K\{x,y\}/\langle xy-x\rangle$ with semi-graduation $B_n:=_K\langle y^kx^{n-k}|0\leq k\leq
n\rangle$, $n\geq 0$, is a $FSG$ algebra and it is not finitely graded generated in degree $1$.
Thus, the class of $FSG$ algebras includes properly all finitely graded algebras generated in
degree $1$.

(iv) Any $FSG$ algebra is $\mathbb{N}$-filtered (see Proposition \ref{proposition17.5.5}), but note
that the Weyl algebra $A_1(K)=K\{t,x\}/\langle xt-tx-1\rangle$ is $\mathbb{N}$-filtered but not
$FSG$, i.e., the class of $FSG$ algebras do not coincide with the class of $\mathbb{N}$-filtered
algebras.
\end{remark}

\begin{proposition}
Let $B$ be a $FSG$ algebra over $K$. Then $B_{\geq 1}$ is the unique two-sided maximal ideal of $B$
semi-graded as left ideal.
\end{proposition}
\begin{proof}
From Remark \ref{17.1.2} we have that $B_{\geq 1}$ is a two-sided maximal ideal of $B$, and of
course, semi-graded as left ideal. Let $I$ be another two-sided maximal ideal of $B$ semi-graded as
left ideal; since $I$ is proper, $I\cap B_0=I\cap K=0$; let $x\in I$, then $x=x_0+x_1+\cdots+x_n$,
with $x_i\in B_i$, $1\leq i\leq n$, but since $I$ is semi-graded, $x_i\in I$ for every $i$, so
$x_0=0$, and hence, $x\in B_{\geq 1}$. Thus, $I\subseteq B_{\geq 1}$, and hence, $I=B_{\geq 1}$.
\end{proof}

\subsection{Examples of $FSG$ algebras}\label{section17.9}

\noindent In this subsection we present a wide list of $FSG$ algebras, many of them, within the
class of skew $PBW$ extensions. For the explicit set of generators and relations for these algebras
see \cite{lezamareyes1}, \cite{Reyes2}, \cite{Suarez2}, \cite{Suarez}, \cite{ReyesSuarez2}.

\begin{example}[Skew $PBW$ extensions that are $FSG$ algebras]\label{example17.9.4}
Note that a skew $PBW$ extension of the field $K$ is a $FSG$ algebra if and only if it is constant
and pre-commutative. Thus, we have:

(i) By the classification presented in the tables of Remark \ref{remark1.13}, the following skew
$PBW$ extensions of the field $K$ are $FSG$ algebras: The classical polynomial algebra; the
particular Sklyanin algebra; the universal enveloping algebra of a Lie algebra; the quantum algebra
$\mathcal{U}'(so(3,K))$; the dispin algebra; the Woronowicz algebra; the $q$-Heisenberg algebra;
nine types 3-dimensional skew polynomial algebras; six types of Sridharan enveloping algebra of
3-dimensional Lie algebras.

(ii) Many skew $PBW$ extensions in the first table of Remark \ref{remark1.13} are marked as non
constant, however, reconsidering the ring of coefficients, some of them can be also viewed as skew
$PBW$ extensions of the base field $K$; this way, they are $FSG$ algebras over $K$: The algebra of
shift operators; the algebra of discrete linear systems; the multiplicative analogue of the Weyl
algebra; the algebra of linear partial shift operators; the algebra of linear partial q-dilation
operators.

(iii) In the class of skew quantum polynomials (see \cite{lezamareyes1}) the
\-mul\-ti-\-pa\-ra\-meter quantum affine $n$-space is another example of skew $PBW$ extension of
the field $K$ that is a $FSG$ (actually finitely graded) algebra. In particular, this is the case
for the quantum plane.

(iv) The following skew $PBW$ extensions of the field $K$ are $FSG$ but not finitely graded: The
niversal enveloping algebra of a Lie algebra; the quantum algebra $\mathcal{U}'(so(3,K))$; the
dispin algebra; the Woronowicz algebra; the $q$-Heisenberg algebra; eight of the nine types
3-dimensional skew polynomial algebras; five of the six types of Sridharan enveloping algebra of
3-dimensional Lie algebras.
\end{example}

\begin{example}[$FSG$ algebras that are not skew $PBW$ extensions of $K$]\label{example16.2.5}
The following algebras are $FSG$ but not skew $PBW$ extensions of the base field $K$ (however, in
every example below the algebra is a skew $PBW$ extension of some other subring):

(i) The \textit{Jordan plane} $A$ is the $K$-algebra generated by $x,y$ with relation $yx = xy +
x^2$, so $A=K\{ x, y\}/ \langle yx-xy-x^2\rangle$. $A$ is not a skew $PBW$ extension of $K$, but of
course, it is a $FSG$ algebra over $K$, actually, it is a finitely graded algebra over $K$ (observe
that $A$ can be viewed as a skew $PBW$ extension of $K[x]$, i.e., $A=\sigma(K[x])\langle
y\rangle$).

(ii) The $K$-algebra in Example 1.18 of \cite{Rogalski} is not a skew PBW extension of $K$:
\begin{center}
$A=K\{ x,y,z\}/\langle z^2-xy-yx,zx-xz,zy-yz\rangle$.
\end{center}
However, $A$ is a $FSG$ algebra, actually, it is a finitely graded algebra over $K$ (note that $A$
can be viewed as a skew $PBW$ extension of $K[z]$: $A=\sigma(K[z])\langle x, y\rangle$).

(iii) The following examples are similar to the previous: The homogenized enveloping algebra
$\mathcal{A}(\mathcal{G})$; algebras of diffusion type; the Manin algebra, or more generally, the
algebra $\mathcal{O}_q(M_n(K))$ of quantum matrices; the complex algebra
$V_q(\mathfrak{sl}_3(\mathbb{C}))$; the algebra $\textbf{U}$; the Witten's deformation of
$\mathcal{U}(\mathfrak{sl}(2,K))$; the quantum symplectic space
$\mathcal{O}_q(\mathfrak{sp}(K^{2n}))$; some quadratic algebras in $3$ variables.
\end{example}

\begin{example}[$FSG$ algebras that are not skew $PBW$ extensions]\label{example16.2.6}
The following $FSG$ algebras are not skew $PBW$ extensions:

(i) Consider the Sklyanin algebra with $c\neq 0$ (see Remark \ref{remark1.19}), then $S$ is not a
skew PBW extension, but clearly it is a $FSG$ algebra over $K$.

(ii) The finitely graded $K$-algebra in Example 1.17 of \cite{Rogalski}:
\begin{center}
$B=K\{ x,y\}/\langle yx^2-x^2y,y^2x-xy^2\rangle$.
\end{center}
(iii) Any \textit{monomial quadratic algebra}
\begin{center}
$B=K\{x_1,\dots,x_n\}/\langle x_ix_j, (i,j)\in S\rangle$,
\end{center}
with $S$ any finite set of pairs of indices (\cite{Polishchuk}).

(iv) $B=K\{w,x,y,u\}/\langle yu, ux-xu,uw\rangle$ (\cite{Phan1}).

(v) $B=K\{x,y\}/\langle x^2y,y^2x\rangle$ (\cite{Phan1}).

(vi) $B=K\{x,y\}/\langle x^2-xy,yx,y^3\rangle$ (\cite{Cassidy2}).

(vii) $B=K\{w,x,y,z\}/\langle z^2y^2,y^3x^2,x^2w,zy^3x\rangle$ (\cite{Cassidy2}).

(viii) $B=K\{x,y,z\}/\langle x^4,yx^3,x^3z\rangle$ (\cite{Cassidy2}).

(ix) $B=K\{x,y,z\}/\langle xz-zx,yz-zy,x^3z,y^4+xz^3\rangle$ (\cite{Cassidy2}).

(x) $B=K\{x,y,z,w,g\}/\langle y^2z,zx^2+gw^2,y^2w^2,xg-gx,yg-gy,wg-gw,zg-gz\rangle$
(\cite{Cassidy2}).

(xi) $B=K\{x,y\}/\langle x^2y-yx^2,xy^3-y^3x\rangle$ (\cite{Cassidy2}).

(xii) $B=K\{x,y\}/\langle xyx,xy^2x,y^3\rangle$ (\cite{Cassidy2}).
\end{example}

\section{Some invariants associated to $FSG$ algebras}\label{section3}

\noindent Now we will study some invariants associated to finitely semi-graded algebras: The
Hilbert series, the Yoneda algebra and the Poincaré series. The topics that we will consider here
for $FSG$ algebras extend some well known results on finitely graded algebras.

\subsection{The Hilbert series}

In Definition \ref{definition15.2.1} we presented the notion of generalized Hilbert series of a
$FSG$ ring. We will prove next that if $B$ is a $FSG$ algebra over a field $K$, then $Gh_B(t)$ is
well-defined, i.e., it does not depend on the semi-graduation (compare with Remark
\ref{remark17.5.11}). This theorem was proved recently by Jason Bell and James J. Zhang in
\cite{ZhangJ4} for connected graded algebras finitely generated in degree 1, we will apply the
Bell-Zhang result to our semi-graded algebras.

\begin{theorem}[\cite{ZhangJ4}]\label{theorem16.1.8}
Let $A$ and $B$ be connected graded algebras finitely generated in degree $1$. Then, $A\cong B$ as
$K$-algebras if and only if $A\cong B$ as graded algebras.
\end{theorem}

\begin{corollary}[\cite{ZhangJ4}]\label{corollary16.1.9}
Let $A$ be a connected graded algebra finitely generated in degree $1$. If $A$ has two graduations
$A=\bigoplus_{n\geq 0}A_{n}=\bigoplus_{n\geq 0}B_{n}$, then there exists an algebra automorphism
$\phi: A\to A$ such that $\phi(A_n)=B_n$ for every $n\geq 0$. In particular, $\dim_K A_n=\dim_K
B_n$ for every $n\geq 0$, and the Hilbert series of $A$ is well-defined. Moreover, if
$Aut(A)=Aut_{\textbf{gr}}(A)$, then $A_n=B_n$ for every $n\geq 0$.
\end{corollary}

We will prove that the generalized Hilbert series of $FSG$ algebras is well-defined.

\begin{proposition}\label{proposition19.3.13}
If $B$ is a $FSG$ algebra, then $Gr(B)$ is a connected graded algebra finitely generated in degree
$1$.
\end{proposition}
\begin{proof}
This is a direct consequence of part (iii) of Proposition \ref{proposition17.5.5}.
\end{proof}

\begin{theorem}\label{theorem16.1.12}
Let $B$ and $C$ be $FSG$ algebras over the field $K$. If $\phi:B\to C$ is a homogeneous isomorphism
of $K$-algebras, then $Gr(B)\cong Gr(C)$ as graded algebras.
\end{theorem}
\begin{proof}
From the previous proposition we know that $Gr(B)$ and $Gr(C)$ are connected graded algebras
finitely generated in degree $1$; according to Theorem \ref{theorem16.1.8} we only have to show
that $Gr(B)$ and $Gr(C)$ are isomorphic as $K$-algebras. For every $n\geq 0$ we have the
homomorphism of $K$-vector spaces $\widetilde{\phi_n}:Gr(B)_n\to Gr(C)$, $b_n\mapsto c_n$, with
$\phi(b_n):=c_n$ (observe that $Gr(B)_n\cong B_n$ and $Gr(C)_n\cong C_n$ as $K$-vector spaces);
from this we obtain a homomorphism of $K$-vector spaces $\widetilde{\phi}:Gr(B)\to Gr(C)$ such that
$\widetilde{\phi}\circ \mu_n=\widetilde{\phi_n}$, for every $n\geq 0$, where $\mu_n:Gr(B)_n\to
Gr(B)$ is the canonical injection. Considering $\varphi:=\phi^{-1}$ we get a homomorphism of
$K$-vector spaces $\widetilde{\varphi}:Gr(C)\to Gr(B)$ such that $\widetilde{\varphi}\circ
\nu_n=\widetilde{\varphi_n}$, for every $n\geq 0$, where $\upsilon_n:Gr(C)_n\to Gr(C)$ is the
canonical injection. But observe that $\widetilde{\phi}\circ \widetilde{\varphi}=i_{Gr(C)}$ and
$\widetilde{\varphi}\circ \widetilde{\phi}=i_{Gr(B)}$. In fact,
$\widetilde{\varphi}\widetilde{\phi}(b_n)=
\widetilde{\varphi}\widetilde{\phi}(\mu_n(b_n))=\widetilde{\varphi}\widetilde{\phi_n}(b_n)=
\widetilde{\varphi}(c_n)=\widetilde{\varphi}\upsilon_n(c_n)=
\widetilde{\varphi_n}(c_n)=\phi^{-1}(c_n)=b_n$. In a similar way we can prove the first identity.
It is obvious that $\widetilde{\phi}$ is multiplicative.
\end{proof}

\begin{corollary}\label{corollary3.10}
Let $B$ be a $FSG$ algebra. If $B$ has two semi-graduations $A=\bigoplus_{n\geq
0}B_{n}=\bigoplus_{n\geq 0}C_{n}$, then $\dim_K B_n=\dim_K C_n$ for every $n\geq 0$, and the
generalized Hilbert series of $B$ is well-defined. Moreover, $Gh_B(t)=h_{Gr(B)}(t)$.
\end{corollary}
\begin{proof}
We consider the identical isomorphism $i_B: B\to B$, by Theorem \ref{theorem16.1.12}, there exists
an isomorphism of graded algebras $\phi:Gr_1(B)\to Gr_2(B)$, where $Gr_1(B)$ is the graded algebra
associated to the semi-graduation $\{B_n\}_{n\geq 0}$ and $Gr_2(B)$ is the graded algebra
associated to $\{C_n\}_{n\geq 0}$; from the proof of Corollary \ref{corollary16.1.9} we know that
$\dim_K (Gr_1(B)_n)=\dim_K (Gr_2(B)_n)$ for every $n\geq 0$, but from the part (iii) of Proposition
\ref{proposition17.5.5}, $Gr_1(B)_n\cong B_n$ and $Gr_2(B)_n\cong C_n$, moreover, these
isomorphisms are $K$-linear, so $\dim_K B_n=\dim_K C_n$ for every $n\geq 0$.
\end{proof}

\begin{corollary}\label{corollary16.3.16}
Each of the algebras presented in Examples \ref{example17.9.4}, \ref{example16.2.5} and
\ref{example16.2.6} have generalized Hilbert series well-defined. In addition, let
$A=\sigma(K)\langle x_1,\dots,x_n\rangle$ be a skew $PBW$ extension of the field $K$; if $A$ is a
$FSG$ algebra, then the generalized Hilbert series is well-defined and given by
\begin{equation*}
Gh_A(t)=\frac{1}{(1-t)^n}.
\end{equation*}
\end{corollary}
\begin{proof}
Direct consequence of the previous corollary and Theorem \ref{17.5.10}.
\end{proof}

\begin{example}\label{example16.3.7}
In this example we show that the condition (iv) in Definition \ref{definition16.1.1} is necessary
in order to the generalized Hilbert series of $FSG$ algebras be well-defined. Let $\mathcal{L}$ be
the 3-dimensional (Heisenberg) Lie algebra that has $K$-basis $\{ x, y, z\}$ with Lie bracket
\begin{center}
$[x,y]=z$, $[x,z]=0$, $[y,z]=0$.
\end{center}
The universal enveloping algebra $\mathcal{U}(\mathcal{L})$ is connected graded with $\deg x= \deg
y=1$, $\deg z=2$. With this grading, the homogeneous component of degree 1 of
$\mathcal{U}(\mathcal{L})$ is $Kx+ Ky$. Thus, $\mathcal{U}(\mathcal{L})$ is not generated in degree
1, i.e., with this grading, $\mathcal{U}(\mathcal{L})$ can not be viewed as $FSG$ algebra. In this
case the generalized Hilbert series is
\begin{center}
$\frac{1}{(1-t)^2 (1-t^2)}$.
\end{center}
On the other hand, $\mathcal{U}(\mathcal{L})$ is $FSG$ by setting $\deg x= \deg y= \deg z=1$.
According to Corollary \ref{corollary16.3.16}, in this case the generalized Hilbert series is
\begin{center}
$\frac{1}{(1-t)^3}$.
\end{center}
\end{example}

\subsection{The Yoneda algebra}

The collection ${\rm SGR}-B$ of semi-graded modules over $B$ is an abelian category, where the
morphisms are the homogeneous $B$-homomorphisms; $K$ is an object of this category with the trivial
semi-graduation given by $K_0:=K$ and $K_n:=0$ for $n\neq 0$. We can associate to $B$ the
\textit{Yoneda algebra} defined by
\begin{equation}\label{equation18.3.5}
E(B):=\bigoplus_{i\geq 0}Ext_B^i(K,K);
\end{equation}
recall that in any abelian category the $Ext_B^i(K,K)$ groups can be computed either by projective
resolutions of $K$ or by extensions of $K$. Here we will take in account both equivalent
interpretations; the first one will be used in the proof of Theorem \ref{corollary18.3.20}. For the
second interpretation (see \cite{Weibel} ), the groups $Ext_B^i(K,K)$ are defined by equivalence
classes of exact sequences of finite length with semi-graded $B$-modules and homogeneous
$B$-homomorphisms from $K$ to $K$:
\begin{equation*}
\xi: 0 \rightarrow K\rightarrow X_i\rightarrow \cdots \rightarrow X_1\rightarrow K\rightarrow 0;
\end{equation*}
the addition in $Ext_B^i(K,K)$ is the \textit{Baer sum} (see \cite{Weibel}, Section 3.4):

\begin{equation*}
\xi: 0 \rightarrow K\rightarrow X_i\rightarrow \cdots \rightarrow X_1\rightarrow K\rightarrow 0,
\end{equation*}
\begin{equation*}
\chi: 0 \rightarrow K\rightarrow X_i'\rightarrow \cdots \rightarrow X_1'\rightarrow K\rightarrow 0,
\end{equation*}
\begin{equation*}
[\xi]\boxplus[\chi]: 0 \rightarrow K\rightarrow Y_i\rightarrow X_{i-1}\oplus X_{i-1}'\rightarrow
\cdots \rightarrow X_2\oplus X_2' \rightarrow Y_1 \rightarrow K\rightarrow 0,
\end{equation*}
where $Y_1$ is the pullback of homomorphisms $X_1\rightarrow K$ and $X_1'\rightarrow K$, and $Y_i$
is the pushout of $K\rightarrow X_i$ and $K\rightarrow X_i'$. The zero element of $Ext_B^i(K,K)$ is
the class of any split sequence $\xi$.

The product in $E(B)$ is given by concatenation of sequences:
\begin{align*}
Ext_B^i(K,K)\times Ext_B^j(K,K) & \to Ext_B^{i+j}(K,K)\\
([\chi],[\xi])& \mapsto [\chi][\xi]:=[\chi\xi],
\end{align*}
where
\begin{equation*}
\xi: 0 \rightarrow K\rightarrow X_i\rightarrow \cdots \rightarrow X_1\rightarrow K\rightarrow 0,
\end{equation*}
\begin{equation*}
\chi: 0 \rightarrow K\rightarrow X_j'\rightarrow \cdots \rightarrow X_1'\rightarrow K\rightarrow 0,
\end{equation*}
\begin{equation*}
\chi\xi: 0 \rightarrow K\rightarrow X_j'\rightarrow \cdots \rightarrow X_1'\rightarrow
X_i\rightarrow \cdots \rightarrow X_1\rightarrow K\rightarrow 0.
\end{equation*}
Note that the unit of $E(B)$ is the equivalence class of $0\to K\xrightarrow{i_K}K\to 0$.

Thus, $E(B)=\bigoplus_{i\geq 0}E^i(B)$ is a connected $\mathbb{N}$-graded algebra, where
$E^i(B):=Ext_B^i(K,K)$ is a $K$-vector space. Observe that definition (\ref{equation18.3.5})
extends the usual notion of Yoneda algebra of graded algebras.

\subsection{The Poincaré series}

Another invariant that we want to consider is the Poincaré series; let $B$ be a $FSG$ algebra; as
we observed above, $E(B)$ is connected and graded; if $E(B)$ is finitely generated, then $E(B)$ is
locally finite, and hence, the \textit{Poincaré series} of $B$ is defined as the Hilbert series of
$E(B)$, i.e.,
\begin{equation}\label{equation19.3.5}
P_B(t):=\sum_{n=0}^\infty(\dim_{K} Ext_B^n(K,K))t^n.
\end{equation}
By Corollary \ref{corollary16.1.9}, $P_B(t)$ is well-defined if $E(B)$ is generated in degree $1$.
In our next theorem we will show that in this case $P_B(t)$ can be also defined by the $Tor$ vector
spaces (compare with \cite{Backelin1} and \cite{Froberg}).

\begin{theorem}\label{corollary18.3.20}
Let $B$ be a $FSG$ algebra such that $K$ has a $B$-free homogeneous resolution and $E(B)$ is
finitely generated. Then, for every $n\geq 0$
\begin{center}
$Tor_n^B(K,K)\cong Ext_B^n(K,K)$.
\end{center}
\end{theorem}
\begin{proof}
Consider the $B$-free homogeneous resolution of $K$ {\small
\begin{equation*}
\cdots \rightarrow B^{(X_n)}\xrightarrow{\beta_{n}} B^{(X_{n-1})}\xrightarrow{\beta_{n-1}} \cdots
\xrightarrow{\beta_2} B^{(X_1)}\xrightarrow{\beta_1} B^{(X_0)}\xrightarrow{\beta_0} K\to 0,
\end{equation*}}
the $Ext^n_B(K,K)$ and the $Tor_n^B(K,K)$ spaces can be computed applying $Hom_B(-,K)$ and
$K\otimes_B -$, respectively, {\tiny
\begin{equation*}
0\to  Hom_{B}(K,K)\xrightarrow{\beta_0^*}\cdots
\xrightarrow{\beta_{n-1}^*}Hom_{B}(B^{(X_{n-1})},K)\xrightarrow{\beta_{n}^*}Hom_{B}(B^{(X_{n})},K)\to
\cdots
\end{equation*}
\begin{equation*}
\cdots \rightarrow K\otimes_B B^{(X_n)}\xrightarrow{i_K\otimes \beta_{n}} K\otimes_B
B^{(X_{n-1})}\xrightarrow{i_K \otimes \beta_{n-1}} \cdots \xrightarrow{i_K\otimes \beta_1}
K\otimes_B B^{(X_0)}\xrightarrow{i_K \otimes \beta_0} K\to 0;
\end{equation*}}
as in Remark \ref{17.1.2}, it is easy to show that $K\cong Hom_B(K,K)$ as $B$-bimodules, whence we
have
\begin{center}
$Hom_B(B^{(X_n)},K)\cong Hom_B(B^{(X_n)}, Hom_B(K,K))\cong Hom_B(K\otimes_B B^{(X_n)}, K)$,
\end{center}
but $B_{\geq 1}(K\otimes_B B^{(X_n)})=0$, then from the previous isomorphism and considering that
$Tor^B_n(K,K)^*=Hom_{K}(Tor^B_n(K,K),K)$, we get
\begin{center}
$Tor^B_n(K,K)^*\cong Ext^n_B(K,K),$
\end{center}
but since $E(B)$ is finitely generated, then $\dim_K Ext^n_B(K,K)<\infty$, and from this
\begin{center}
$Tor_n^B(K,K)\cong Tor^B_n(K,K)^*\cong Ext_B^n(K,K)$.
\end{center}
\end{proof}

\begin{corollary}\label{corollary19.3.22}
Let $B$ be a $FSG$ algebra such that $K$ has a $B$-free homogeneous resolution and $E(B)$ is
finitely generated in degree $1$, then $P_B(t)$ is well-defined and it is also given by
\begin{equation}
P_B(t)=\sum_{n=0}^\infty(\dim_{K} Tor_n^B(K,K))t^n.
\end{equation}
\end{corollary}
\begin{proof}
This follows from (\ref{equation19.3.5}) and the previous theorem.
\end{proof}

\section{Koszulity}

\noindent Koszul algebras were defined by Stewart B. Priddy in \cite{Priddy}. Later in 2001, Roland
Berger in \cite{Berger2} introduces a generalization of Koszul algebras which are called
\emph{generalized Koszul algebras or $N$-Koszul algebras}. The $2$-Koszul algebras of Roland Berger
are the Koszul algebras of Priddy (for the definition of Koszul algebras adopted in this paper see
Remark \ref{remark4.6}). $N$-Koszul algebras are finitely graded where all generators of the ideal
$I$ of relations are homogeneous and have the same degree $N\geq 2$. In 2008 Thomas Cassidy and
Brad Shelton (\cite{Cassidy2}) generalize the $N$-Koszul algebras introducing the $\mathcal{K}_2$
algebras; these type of algebras accept that the generators of $I$ have different degrees, but
again all generators are homogeneous since the $\mathcal{K}_2$ algebras are graded. Later, Phan in
\cite{Phan} extended this notion to $\mathcal{K}_m$ algebras for any $m\geq1$.

In this section we study the semi-graded version of Koszulity, and for this purpose we will follow
the lattice interpretation of this notion (see \cite{Backelin1}, \cite{BackelinFroberg},
\cite{Berger2}, \cite{Froberg}, \cite{Polishchuk}).

\subsection{Semi-graded Koszul algebras}

Recall that a \emph{lattice} is a collection $L$ endowed with two idempotent commutative and
associative binary operations $\wedge, \vee: L \times L\to L$ satisfying the following
\emph{absorption identities}: $a \wedge (a \vee b)= a$, $(a\wedge b)\vee b=b$. A sublattice of a
lattice $L$ is a non empty subset of $L$ closed under $\wedge$ and $\vee$. A lattice is called
\emph{distributive} if it satisfies the following distributivity identity: $a\wedge (b \vee c) =
(a\wedge b) \vee (a\wedge c)$. If $X\subseteq L$, the sublattice \emph{generated} by $X$, denoted
$[X]$, consists of all elements of $L$ that can be obtained from the elements of $X$ by the
operations $\wedge$ and $\vee$. We will say that $X$ is \emph{ distributive} if $[X]$ is a
distributive lattice. The \emph{$($direct$)$ product} of the family of lattices
$\{L_\omega\}_{\omega\in \Omega}$ is defined as follow:
\[
\prod_{\Omega}L_\omega:=(\prod_{\Omega}L_\omega,\wedge, \vee),
\]
which is the cartesian product with $\wedge$ and $\vee$ operating component-wise. A
\emph{semidirect product} of the family $\{L_\omega\}_{\omega\in \Omega}$ is a sublattice $L$ of
$\prod_{\Omega}L_\omega$ such that for every $\omega_0\in \Omega$, the composition
\begin{center}
$L\hookrightarrow \prod_{\Omega}L_\omega\twoheadrightarrow L_{\omega_0}$
\end{center}
is surjective.

\begin{proposition}[\cite{Backelin1}]\label{proposition19.4.1}
If $L$ is a semidirect product of the family $\{L_\omega\}_{\omega\in \Omega}$, then $L$ is
distributive if and only if for all $\omega\in \Omega$, $L_\omega$ is distributive.
\end{proposition}

Let $K$ be a field and $V$ be a $K$-vector space, the set $L(V)$ of all its linear subspaces is a
lattice with respect to the operations of sum and intersection.

\begin{proposition}[\cite{Polishchuk}]\label{proposition19.4.2}
Let $V$ be a vector space and $X_1,\dots, X_n\subseteq V$ be a finite collection of subspaces of
$V$. The following conditions are equivalent:
\begin{enumerate}
\item[\rm (i)]The collection $X_1,\dots, X_n$ is distributive.
\item[\rm (ii)]There exists a basis $\mathcal{B}:=\{\omega_i\}_{i\in \mathcal{C}}$ of $V$ such that
each of the subspaces $X_i$ is the linear span of a set of vectors $\omega_i$.
\item[\rm (iii)]There exists a basis $\mathcal{B}$ of $V$ such that $\mathcal{B}\cap X_i$ is a
basis of $X_i$, for every $1\leq i \leq n$.
\end{enumerate}
\end{proposition}

With the previous elementary facts about lattices, we have the following notions associated to any
$FSG$ algebra presented as in (\ref{equ17.1.1}) (compare with \cite{Backelin1}).

\begin{definition}
Let $B=K\{ x_1,\dots,x_n\}/I$ be a $FSG$ algebra. The lattice associated to $B$ is the sublattice
$L(B)$ of subspaces of the free algebra $F:=K\{ x_1,\dots,x_n\}$ generated by $\{F_{\geq
1}^sI^gF_{\geq 1}^h|s,g,h\geq 0\}$. For any integer $j\geq 2$, the $j-th$ lattice associated to $B$
is defined by
\[
L_j(B):=[\{F_sI_gF_h | s,h\geq 0, g\geq 2, s+g+h=j\}]\subset \{\text{subspaces of}\, F_j; \cap,
+\},
\]
where $F_sI_gF_h$ is the subspace of $F_j$ consisting of finite sums of elements of the form $abc$,
with $a\in F_s, b\in I_g, c\in F_h$, and
\[
I_g:=\{a_g\in F_g|a_g\, \text{is the $g-th$ component of some element in $I$}\}.
\]
\end{definition}
For any two-sided ideal $H$ of $F$, the $K$-subspace $H_g$ is defined similarly. From now on in
this section we will denote $F:=K\{ x_1,\dots,x_n\}$.

\begin{theorem}\label{theorem19.4.4}
Let $B=K\{ x_1,\dots,x_n\}/I$ be a $FSG$ algebra with $I=\langle b_1,\dots,b_m\rangle$ such that
$b_i\in F_{\geq 1}$ for $1\leq i\leq m$. Then $L(B)$ is a semidirect product of the family of
lattices
\[
\{L_j(B)\cup \{0,F_j\}_{j\geq 2}\}\cup \{\{0,K\},\{0,F_1\}\}.
\]
In particular, $L(B)$ is distributive if and only if for all $j\geq 2$, $L_j(B)$ is distributive.
\end{theorem}
\begin{proof}
The proof of Lemma 2.4 in \cite{Backelin1} can be easy adapted.

\textit{Step 1}. For any $j\geq 2$ and any $X\in L_j(B)$ we have $0\subseteq X\subseteq F_j$. So
$L_j(B)\cup \{0,F_j\}$ is in fact a lattice.

\textit{Step 2}. If $s\geq 0$, $g\geq 1$, $h\geq 0$ and $j\geq 2+s+h$, then
\[
(F_{\geq 1}^sI^gF_{\geq 1}^h)_j=F_s(I^g)_{j-s-h}F_.
\]
We only have to prove that $(F_{\geq 1}^sI^gF_{\geq 1}^h)_j\subseteq F_s(I^g)_{j-s-h}F_h$ since the
other containment is trivial. Recall that one element of $(F_{\geq 1}^sI^gF_{\geq 1}^h)_j$ is the
$j-th$ component of some element of $F_{\geq 1}^sI^gF_{\geq 1}^h$; let $z_j\in (F_{\geq
1}^sI^gF_{\geq 1}^h)_j$, then there exists $y\in F_{\geq 1}^sI^gF_{\geq 1}^h$ such that $z_j$ is
the $j-th$ component of $y$; the element $y$ is a finite sum of elements of the form $abc$, with
$a\in F_{\geq 1}^s=F_{\geq s}$, $b\in I^g$ and $c\in F_{\geq 1}^h=F_{\geq h}$, so the $j-th$
component of $y$ is a sum of the $j-th$ components of elements of the form $a_kba_t$, with $k\geq
s$, $b\in I^g$ and $t\geq h$, but since $F_k=F_sF_{k-s}$ for $k\geq s$ and $F_t=F_{t-h}F_h$ for
$t\geq h$, then the $j-th$ component of $a_kba_t$ is the $j-th$ component
$a_s(a_{k-s}ba_{t-h})a_h$, i.e., it is an element of $F_s(I^g)_{j-s-h}F_h$.

\textit{Step 3}. For $g\geq 1$ and $j\geq 2$,
\[
(I^g)_j=\sum F_{k_0} I_{l_1}F_{k_1}I_{l_2}\cdots F_{k_{g-1}}I_{l_g}F_{k_g},
\]
where the sum is taken over all relevant $k_0,\dots,k_g,l_1,\dots,l_g$ such that $\sum_m k_m+\sum_n
l_n=j$. Indeed, if $p\in I^g$, then $p$ is a finite sum of elements of the form
$a^{(0)}p_1a^{(1)}p_2\cdots a^{(g-1)}p_ga^{(g)}$, with $a^{(r)}\in F$, $p_i\in \{b_1,\dots,b_m\}$,
$0\leq r\leq g$, $1\leq i\leq g$.

\textit{Step 4}. For any $g\geq 2$ and any $2g+1$ non-negative integers $k_0,\dots,k_g,
l_1,\dots,l_g$ we have
\[
F_{k_0} I_{l_1}F_{k_1}I_{l_2}\cdots F_{k_{g-1}}I_{l_g}F_{k_g}=\bigcap_{a=1}^g
F_{k_0+l_1+\cdots+k_{a-1}}I_{l_a}F_{k_a+\cdots+k_g}.
\]
In fact, let $q=a_0p_1a_1\cdots p_ga_g\in F_{k_0} I_{l_1}F_{k_1}I_{l_2}\cdots
F_{k_{g-1}}I_{l_g}F_{k_g}$, with $a_r\in F_{k_r}$, $p_i\in I_{l_i}$, $0\leq r\leq g$, $1\leq i\leq
g$, then $q\in F_{k_0+l_1+\cdots+k_{a-1}}I_{l_a}F_{k_a+\cdots+k_g}$ for every $1\leq a\leq g$; the
converse follows from the fact that for any $a\in F-\{0\}$ homogeneous with $a=bc=de$, then
$b,c,d,e$ are homogeneous; in addition, if $b\in F_k$, $d\in F_t$ with $t\geq s$, then there is $f$
such that $a=bfe$, $d=bf$ and $c=fe$.

\textit{Step 5}. For any $s\geq 0$, $g\geq 1$, $h\geq 0$ and $j<1+s+h$ we have $(F_{\geq
1}^sI^gF_{\geq 1}^h)_j=0$ since $b_i\in F_{\geq 1}$ for $1\leq i\leq m$; likewise, for $j<g$,
$(I^g)_j=0$.

From these steps, $L(B)$ is a sublattice of the product of the given family, i.e.,
\[
L(B)\hookrightarrow \{0,K\}\times \{0,F_1\}\times (\prod_{j\geq 2}L_j(B)\cup \{0,F_j\}).
\]
Finally, fix $j\geq 2$, then $L(B)\to L_j(B)\cup \{0,F_j\}$ is a lattice surjective map since: (a)
$(I^g)_j=0$ if $j< g$; (b) $(F_{\geq 1}^s)_j=F_j$ if $j\geq s$; (c) if $s,h\geq 0$, $g\geq 2$ and
$s+g+h=j$, then $F_sI_gF_h=(F_{\geq 1}^sI^gF_{\geq 1}^h)_j$. The cases $j=0,1$ can be proved by the
same method. Thus, $L(B)$ is a semidirect product of the given family.
\end{proof}

\begin{definition}\label{definition19.4.5}
Let $B=K\{ x_1,\dots,x_n\}/I$ be a $FSG$ algebra. We say that $B$ is semi-graded Koszul, denoted
$S\mathcal{K}$, if $B$ satisfies the following conditions:
\begin{enumerate}
\item[\rm (i)]$B$ is finitely presented with $I=\langle b_1,\dots,b_m\rangle$ and
$b_i\in F_{\geq 1}$ for $1\leq i\leq m$.
\item[\rm (ii)]$L(B)$ is distributive.
\end{enumerate}
\end{definition}

\begin{remark}\label{remark4.6}
(i) In the present paper we adopt the following definition of Koszul algebras (see
\cite{Backelin1}, \cite{BackelinFroberg}, \cite{Berger2}, \cite{Froberg}, \cite{Polishchuk}). Let
$B$ be a $K$-algebra; it is said that $B$ is \textit{Kozul} if $B$ satifies the following
conditions: (a) $B$ is $\mathbb{N}$-graded, connected, finitely generated in degree one; (b) $B$ is
quadratic, i.e., the ideal $I$ in $(\ref{equ17.1.1})$ is finitely generated by homogeneous elements
of degree $2$; (c) $L(B)$ is distributive.

(ii) From (i) it is clear that any Koszul algebra is $S\mathcal{K}$. Many examples of skew $PBW$
extensions are actually Koszul algebras. In \cite{Suarez} and \cite{ReyesSuarez2} was proved that
the following skew $PBW$ extensions are Koszul algebras: The classical polynomial algebra; the
particular Sklyanin algebra; the  multiplicative analogue of the Weyl algebra; the algebra of
linear partial $q$-dilation operators; the \-mul\-ti-\-pa\-ra\-meter quantum affine $n$-space, in
particular, the quantum plane; the 3-dimensional skew polynomial algebra with $|\{\alpha, \beta,
\gamma\}|=3$; the Sridharan enveloping algebra of 3-dimensional Lie algebra with
$[x,y]=[y,z]=[z,x]=0$; The Jordan plane; algebras of diffusion type; $\mathcal{A}(\mathcal{G})$;
the algebra $\textbf{U}$; the Manin algebra, or more generally, the algebra $\mathcal{O}_q(M_n(K))$
of quantum matrices; some quadratic algebras in $3$ variables.
\end{remark}

The next theorem gives a wide list of $\mathcal{SK}$ algebras within the class of skew $PBW$
extensions. If at least one of the constants $a_{ij}^{(k_{i,j})}$ is non zero, then the algebra is
not Koszul but it is $\mathcal{SK}$.

\begin{theorem}\label{theorem19.4.13}
If $A$ is a skew $PBW$ extension of a field $K$ with presentation $A=K\{x_1,\ldots,x_n\}/I$, where
\begin{center}
$I=\langle x_jx_i-c_{ij}x_ix_j-a_{ij}^{(k_{i,j})}x_{k_{i,j}}|c_{ij},a_{ij}^{(k_{i,j})}\in K,
c_{ij}\neq 0, 1\leq j<i\leq n\rangle$,
\end{center}
then $A$ is $\mathcal{SK}$.
\end{theorem}
\begin{proof}
Note that $A$ is a $FSG$ algebra. Let $F:=K\{x_1,\dots,x_n\}$, $N:=\{x_1,\ldots,x_n\}$, and
$J:=\{k_{i,j}\in\{1,\ldots,n\}|a_{k_{i,j}}\neq 0, 1\leq i<j\leq n\}$. We are going to show that
$L_m(A)$ is distributive lattice for $m\geq 2$.

If $|J|=n$, we define

\begin{align*}
\mathcal{B}_m:=\left(\bigcup_{r=1}^{m}D_r^{(m)}\right),
\end{align*}
where
$$D_r^{(m)}:=\{a_1\cdots a_{r-1}x_ia_{r+1}\cdots a_m|a_t\in N, t=1,\ldots,r-1,r+1,\ldots,n;1\leq i\leq n\};$$
$\mathcal{B}_m$ is a basis of $F_m$. Now, consider $F_sI_gF_h\leq F_m$ with $s,h\geq 0$, $g\geq 2$
and $s+g+h=m$. Since $F_sI_gF_h$ is generated by $D_{s+1}^{(m)},\ldots, D_{s+g}^{(m)}$, then
$F_sI_gF_h\cap \mathcal{B}_m=\bigcup_{r=s+1}^{s+g}D_r^{(m)}$, which is a basis of $F_sI_gF_h$.

If $|J|=n-1$, define

\begin{align*}
\mathcal{B}_m:=\left(\bigcup_{r=1}^{m}D_r^{(m)}\right)\cup \{x_l^m\},
\end{align*}
where $l\notin J$, and
$$D_r^{(m)}:=\{a_1\cdots a_{r-1}x_ia_{r+1}\cdots a_m|a_t\in N, t=1,\ldots,r-1,r+1,\ldots,n;i\in J\};$$
again $\mathcal{B}_m$ is a basis of $F_m$. As before, consider $F_sI_gF_h\leq F_m$ with $s,h\geq
0$, $g\geq 2$ and $s+g+h=m$; since $F_sI_gF_h$ is generated by $D_{s+1}^{(m)},\ldots,
D_{s+g}^{(m)}$, then $F_sI_gF_h\cap \mathcal{B}_m=\bigcup_{r=s+1}^{s+g}D_r^{(m)}$, which is a basis
of $F_sI_gF_h$.

If $|J|\leq n-2$, we define

\begin{align*}
\mathcal{B}_m:=\left(\bigcup_{r=1}^{m-1}B_r^{(m)}\right)\cup\left(\bigcup_{r=1}^{m-1}C_r^{(m)}\right)\cup
\left(\bigcup_{r=1}^m D_r^{(m)}\right)\cup E,
\end{align*}
where
\begin{align*}
\scriptstyle B_r^{(m)}:=\{a_1\cdots a_{r-1}x_jx_ia_{r+2}\cdots a_m|a_t\in N; t=1,2,\ldots,r-1,r+2,\ldots, m;i,j\notin J; i<j\},\\
\scriptstyle C_r^{(m)}:=\{a_1\cdots a_{r-1}(x_ix_j-c_{ij}x_jx_i)a_{r+2}\cdots a_m|a_t\in N; t=1,2,\ldots,r-1,r+2,\ldots, m;i,j\notin J; i<j\},\\
\scriptstyle D_r^{(m)}:=\{a_1\cdots a_{r-1}x_la_{r+1}\cdots a_m|a_t\in N, t=1,\ldots,r-1,r+1,\ldots,n;l\in J\},\\
\scriptstyle E=\{x_i^m|i\notin J\}.
\end{align*}
$\mathcal{B}_m$ is a basis of $A_m$; consider $F_sI_gF_h\leq F_m$ with $s,h\geq 0$, $g\geq 2$ and
$s+g+h=m$; since $F_sI_gF_h$ is generated by $C_{s+1}^{(m)},\ldots,C_{s+g-2}^{(m)},
D_{s+1}^{(m)},\ldots, D_{s+g}^{(m)}$, then $F_sI_gF_h\cap
\mathcal{B}_m=\bigcup_{r=s+1}^{s+g}C_r^{(m)}\cup\left(\bigcup_{r=s+1}^{s+g-2}D_r\right)$, which is
a basis of $F_sI_gF_h$.
\end{proof}

\begin{example}\label{example4.8}
(i) The following algebras satisfy the conditions of the previous theorem, and hence, they are
$\mathcal{SK}$ (but not Koszul): The dispin algebra $\mathcal{U}(osp(1,2))$; the $q$-Heisenberg
algebra; the quantum algebra $\mathcal{U}'(\mathfrak{so}(3,K))$; the Woronowicz algebra
$\mathcal{W}_{\nu}(sl(2,K))$; the algebra $S_h$ of shift operators; the algebra $D$ for
multidimensional discrete linear systems; the algebra of linear partial shift operators.

(ii) The following algebras do not satisfy the conditions of the previous theorem, but by direct
computation we proved that the lattice $L(B)$ is distributive, so they are $\mathcal{SK}$ (but not
Koszul): The algebra $V_q(\mathfrak{sl}_3(\mathbb{C}))$; the Witten's deformation of
$\mathcal{U}(\mathfrak{sl}(2,K)$; the quantum symplectic space
$\mathcal{O}_q(\mathfrak{sp}(K^{2n}))$.

(iii) In the following table we summarize the examples of skew $PBW$ extensions that Koszul or
semi-graded Koszul:
\begin{table}[htb]
\centering \tiny{
\begin{tabular}{|l|l|l|}\hline
\textbf{$FSG$ algebra} & \textbf{K} & \textbf{SK}\\
\hline \hline Classical polynomial algebra $K[x_1,\dots,x_n]$
 & $\checkmark$ & $\checkmark$\\
\cline{1-3} Some universal enveloping algebras of a Lie algebras $\mathcal{G}$, $\cU(\mathcal{G})$
 & $\times$ & $\checkmark$\\
\cline{1-3} Some Sridharan enveloping algebras of 3-dimensional Lie algebras
 & $\checkmark$ & $\checkmark$\\
\cline{1-3} Particular  Sklyanin algebra
 & $\checkmark$ & $\checkmark$\\
\cline{1-3} Homogenized enveloping algebra $\mathcal{A}(\mathcal{G})$
 & $\checkmark$ & $\checkmark$\\
\cline{1-3} Algebra of shift operators $S_h$
 & $\times$ & $\checkmark$\\
\cline{1-3} Algebra of discrete linear systems
$K[t_1,\dotsc,t_n][x_1;\sigma_1]\dotsb[x_n;\sigma_n]$
 & $\times$ & $\checkmark$\\
\cline{1-3} Linear partial shift operators $K[t_1,\dotsc,t_n][E_1,\dotsc,E_m]$
 & $\times$ & $\checkmark$\\
\cline{1-3} Linear partial shift operators $K(t_1,\dotsc,t_n)[E_1,\dotsc,E_m]$
 & $\times$ & $\checkmark$\\
\cline{1-3} L. Partial $q$-dilation operators $K[t_1,\dotsc,t_n][H_1^{(q)},\dotsc,H_m^{(q)}]$
 & $\times$ & $\checkmark$\\
\cline{1-3} L. Partial $q$-dilation operators $K(t_1,\dotsc,t_n)[H_1^{(q)},\dotsc,H_m^{(q)}]$
 & $\times$ & $\checkmark$\\
\cline{1-3} Algebras of diffusion type
 & $\checkmark$ & $\checkmark$\\
\cline{1-3} Multiplicative analogue of the Weyl algebra $\cO_n(\lambda_{ji})$
 & $\checkmark$ & $\checkmark$\\
\cline{1-3} Quantum algebra $\cU'(\mathfrak{so}(3,K))$
 & $\times$ & $\checkmark$\\
\cline{1-3} Some 3-dimensional skew polynomial algebras
 & $\checkmark$ & $\checkmark$\\
\cline{1-3} Dispin algebra $\cU(osp(1,2))$
 & $\times$ & $\checkmark$\\
\cline{1-3} Woronowicz algebra $\cW_{\nu}(\mathfrak{sl}(2,K))$
 & $\times$ & $\checkmark$\\
\cline{1-3} Complex algebra $V_q(\mathfrak{sl}(3,\mathbb{C}))$
 &  $\times$ & $\checkmark$\\
\cline{1-3} Algebra \textbf{U}
& $\checkmark$ & $\checkmark$\\
\cline{1-3} Manin algebra $M_q(2)$, $\cO_q(M_2(K))$
 & $\checkmark$ & $\checkmark$\\
\cline{1-3} $q$-Heisenberg algebra \textbf{H}$_n(q)$
 & $\times$ & $\checkmark$\\
\cline{1-3} Witten's deformation of $\cU(\mathfrak{sl}(2,K)$
& $\times$ & $\checkmark$\\
\cline{1-3} Quantum symplectic space $\cO_q(\mathfrak{sp}(K^{2n}))$
& $\times$ & $\checkmark$\\
\cline{1-3} Some quadratic algebras in 3 variables
& $\checkmark$ & $\checkmark$\\
\cline{1-3} Multi-parameter quantum affine n-space
& $\checkmark$ & $\checkmark$\\
\cline{1-3} Jordan plane
& $\checkmark$ & $\checkmark$\\
\cline{1-3}
\end{tabular}}
\caption{$FSG$ algebras: Koszul (K), semi-graded Koszul (SK)}\label{table19.2}
\end{table}

\end{example}
\begin{example}
There exist $FSG$ algebras that are not $\mathcal{SK}$: The algebra $A=K\{x,y\}/\langle x^2-xy,
yx,y^3\rangle$ (see (\cite{Cassidy2}) ) is not a skew $PBW$ extension, but is a $FSG$ algebra. This
algebra satisfies that $L(A)$ is a subdirect product of the family of lattices \begin{align*}
\{L_j(A)\cup \{0, A_j\}\}_{j\geq 2}\cup \{\{0, K\},\{0, A_1\}\},
\end{align*}
but $L_3(A)$ is not distributive. In fact, note that the lattice $L_3(A)$ is generated by $A_1I_2, I_2A_1, I_3$ and
\begin{enumerate}
    \item $A_1I_2$ is $K$-generated by $D=\{x^3-xyx, x^2y-xy^2, yx^2,yxy\}$, and $D$ is $K$-linearly independent, therefore $\dim_K(A_1I_2)=4$.
    \item $I_2A_1$ is $K$-generated by  $C=\{x^3-x^2y, yx^2-yxy, xyx, y^2x\}$, which is $K$-linearly independent, therefore $\dim_K(I_2A_1)=4$.
\end{enumerate}
Now, let us suppose $\mathcal{B}=\{a_1,a_2,\ldots,a_8\}$ be a $K$-basis of $A_3$ such that $X:=\mathcal{B}\cap A_1I_2$ is a basis of $A_1I_2$ and $Y:=\mathcal{B}\cap I_2A_1$ is a basis of $I_2A_1$.

Without loss of generality, suppose that $X=\{a_1,\ldots,a_4\}$, then $yx^2=\lambda_1 a_1+\lambda_2 a_2+\lambda_3 a_3+\lambda_4 a_4$ and $yxy=\beta_1 a_1+\beta_2 a_2+\beta_3 a_3+\beta_4 a_4$ with  $\lambda_i,\beta_i\in K$ for $1\leq i\leq 4$, $\lambda_1\neq \beta_1$ (maybe organizing), $\lambda_1\neq 0$ and $\lambda_j\neq\beta_j$, for some $j=2,3,4$, otherwise, if $\lambda_j=\beta_j$, for $j=2,3,4$, then $yxy-\frac{\beta_1}{\lambda_1}yx^2=0$, which is imposible. So $$yx^2-yxy=(\lambda_1-\beta_1)a_1+(\lambda_2-\beta_2)a_2+(\lambda_3-\beta_3)a_3+(\lambda_4-\beta_4)a_4,$$
with at least $a_1,a_j\in X\cap Y$, consequently \begin{align*}
a_1=&\alpha_1(x^3-xyx)+\alpha_2(x^2y-xy^2)+\alpha_3(yx^2)+\alpha_4(yxy)\\=&\gamma_1(x^2-x^2y)+\gamma_2(yx^2-yxy)+\gamma_3(xyx)+\gamma_4(y^2x),\\
a_j=&\eta_1(x^3-xyx)+\eta_2(x^2y-xy^2)+\eta_3(yx^2)+\eta_4(yxy)\\=&\mu_1(x^2-x^2y)+\mu_2(yx^2-yxy)+\mu_3(xyx)+\mu_4(y^2x),
\end{align*}
with $\alpha_i,\gamma_i,\eta_i,\mu_i\in K$ for $1\leq i\leq 4$. Thus there exist two different
$K$-combinations non-trivial of $C\cup D$ equal to $0$, hence the base $\mathcal{B}$ does not
exist. Thus, $A$ is a $FSG$ algebra but is not $\mathcal{SK}$.

\end{example}

\subsection{Poincaré series of skew $PBW$ extensions}

We conclude computing the Poincaré series of some skew $PBW$ extensions of $K$.

\begin{theorem}\label{theorem19.4.8}
Let $A=\sigma(K)\langle x_1,\dots,x_n\rangle$ be a skew $PBW$ extension of the field $K$ that is a
Koszul algebra, then the Poincaré series of $A$ is well-defined and given by
$P_A(t)=(1+t)^n$.
\end{theorem}
\begin{proof}
Since $A$ is Koszul, then $h_A(t)P_A(-t)=1$ and $E(A)$ is Koszul, whence $E(A)$ is finitely
generated in degree $1$ (see \cite{Berger2}, \cite{Froberg}, or \cite{Polishchuk}); therefore the
theorem follows from Corollaries \ref{corollary16.3.16} and \ref{corollary19.3.22}.
\end{proof}

\begin{example}\label{example4.12}
From Remark \ref{remark4.6} and Theorem \ref{theorem19.4.8}, we present next the Poincaré series of
some skew $PBW$ extensions of the base field $K$:

\begin{table}[htb]
\centering \footnotesize{
\begin{tabular}{|l|l|}\hline
\textbf{$\mathcal{SK}$ algebra} & \textbf{$P_A(t)$} \\
\hline \hline Classical polynomial algebra $K[x_1,\dots,x_n]$
 & $(1+t)^n$\\
\cline{1-2} Some Sridharan enveloping algebras of 3-dimensional Lie algebras
 & $(1+t)^3$\\
\cline{1-2} Particular  Sklyanin algebra
 & $(1+t)^3$\\
\cline{1-2} L. Partial $q$-dilation operators $K[t_1,\dotsc,t_n][H_1^{(q)},\dotsc,H_m^{(q)}]$
 & $(1+t)^{n+m}$\\
\cline{1-2} Multiplicative analogue of the Weyl algebra $\cO_n(\lambda_{ji})$
 & $(1+t)^n$\\
\cline{1-2} Some 3-dimensional skew polynomial algebras
 & $(1+t)^3$\\
\cline{1-2} Multi-parameter quantum affine n-space
& $(1+t)^n$\\
\cline{1-2}
\end{tabular}}
\caption{Poincaré series of some skew $PBW$ extensions of $K$.}\label{table19.1}
\end{table}

\end{example}

\begin{center}
\textbf{Acknowledgements}
\end{center}
The authors are grateful to James Jim Zhang and the referee for valuable corrections, comments and
suggestions.



\begin{thebibliography}{200}
\bibitem{Acosta1}\textbf{Acosta, J. and Lezama, O.},  \textit{Universal property of
skew PBW extensions},  Algebra Discrete Math., 20(1), 2015, 1-12.

\bibitem{Acosta2}\textbf{Acosta, J.P., Lezama, O. and Reyes, M.A.}, \textit{Prime ideals of skew $PBW$ extensions}, Rev. Un. Mat. Argentina,
56(2), 2015, 39-55.

\bibitem{lezamaore}\textbf{Acosta, J.P., Chaparro, C., Lezama, O., Ojeda, I., and Venegas, C.},
\textit{Ore and Goldie theorems for skew $PBW$ extensions}, Asian-Eur. J. Math., 06, 2013, 1350061
[20 pages].
\bibitem{ArtamonovDerivations}\textbf{Artamonov, V.}, \textit{Derivations of Skew PBW-Extensions}, Commun. Math. Stat., 3 (4), 2015, 449-457.
\bibitem{Backelin1}\textbf{Backelin, J.}, \textit{A distributiveness property of augmented algebras and some related
homological results}, Ph. D. Thesis, Stockholm, 1981.
\bibitem{BackelinFroberg}\textbf{Backelin, J. and Fröberg, R.}, \textit{Koszul algebras,
Veronese subrings and rings with linear resolutions}, Rev. Roumaine Math. Pures Appl., 30(2),
85-97, 1985.
\bibitem{ZhangJ4}\textbf{Bell, J., and Zhang, J. J.}, \textit{An isomorphism lemma for graded rings}, Proc. Amer. Math. Soc., 145, 2017, 989-994.
\bibitem{Berger2}\textbf{Berger, R.}, \textit{Koszulity for nonquadratic algebras},
J. Algebra, vol. 239, 2001, 705-734.
\bibitem{Berger}\textbf{Berger, R., Marconnet, N.}, \textit{Koszul and Gorenstein
properties for homogeneous algebras}, Algebr. Represent. Theory, 1, 2006, 67-97.
\bibitem{Cassidy2}\textbf{Cassidy, Th. and Shelton, B.}, \textit{Generalizing the notion of Koszul algebra},
preprint, 2008.
\bibitem{Froberg}\textbf{Fr\"{o}berg, R.}, \textit{Koszul Algebras}, In: Advances in Commutative Ring Theory. Proceedings of the
3rd International Conference, Fez, Lect. Notes Pure Appl. Math. 205, Marcel Dekker, New York, 1999,
337-350.
\bibitem{Gaddis}\textbf{Gaddis, J. D.}, \textit{PBW deformations of Artin-Schelter regular algebras and their homogenizations},
Ph.D. Thesis, The University of Wisconsin-Milwaukee, USA, 2013.
\bibitem{Gallego2}\textbf{Gallego, C. and Lezama, O.}, \textit{Gröbner bases for ideals of skew $PBW$ extensions},
Comm. Algebra, 39, 2011, 50-75.
\bibitem{lezama-gallego-projective}\textbf{Gallego, C. and Lezama, O.}, \textit{Projective modules and Gröbner bases
for skew $PBW$ extensions}, Dissertationes Math., 521, 2017, 1-50.
\bibitem{Ginzburg1} \textbf{Ginzburg, V.}, \textit{Calabi--Yau algebras}, 2006,
arXiv:math.AG/0612139v3.
\bibitem{Kanazawa}\textbf{Kanazawa, A.}, \textit{Non-commutative projective Calabi-Yau schemes}, preprint, 2015.
\bibitem{lezamareyes1}\textbf{Lezama, O. \& Reyes, M.}, {\em Some homological properties of skew $PBW$
extensions}, Comm. Algebra, 42, 2014, 1200-1230.
\bibitem{lezamalatorre}\textbf{Lezama, O. and Latorre, E.}, \textit{Non-commutative algebraic geometry of semi-graded rings},
International Journal of Algebra and Computation, 27 (4), 2017, 361-389.
\bibitem{LezamaHelbert}\textbf{Lezama, O. and Venegas, H.}, \textit{Some homological properties of skew PBW extensions arising in non-commutative algebraic geometry},
Discuss. Math. Gen. Algebra Appl., 37, 2017, 45-57.
\bibitem{Liu-Wu} \textbf{Liu, L.-Y.,  Wu,  Q.-S.}, \textit{Twisted Calabi- Yau property of right coideal subalgebras of quantized enveloping algebras},
J. Algebra, 399, 2014, 1073-1085.
\bibitem{Liu-Wang} \textbf{Liu, L.-Y.,  Wu,  Q.-S., and Wang, S.-Q.},
\textit{Twisted Calabi-Yau propperty of Ore extensions}, J. Noncommut. Geom., 8(2), 2014, 587-609.
\bibitem{Nastasescu}\textbf{Nastasescu, C. and Van Oystaeyen, F.}, \textit{Graded and Filtered Rings and Modules}, Lecture Notes in
Mathematics 758, Springer, 1979.
\bibitem{Phan1}\textbf{Phan, C.}, \textit{The Yoneda algebra of a graded Ore extension}, arXiv: 1002.2318v1 [math.RA].
\bibitem{Phan}\textbf{ Phan, C.}, \textit{Koszul and Generalized Koszul properties for noncommutative graded algebras}, Ph.D. Thesis, University
of Oregon, 2009.
\bibitem{Polishchuk}\textbf{Polishchuk, A., Positselski, C.}, \textit{Quadratic Algebras}, Univ. Lecture Ser., vol. 37, Amer. Math. Soc., 2005.
\bibitem{Priddy}\textbf{Priddy, S.}, \textit{Koszul Resolutions}, Trans. Amer.
Math. Soc., 152, 1970, 39-60.
\bibitem{Reyes2}\textbf{Reyes, M. A.}, \textit{Ring and Module Theoretic Properties of $\sigma$-PBW Extensions}, Ph.D. Thesis,
Universidad Nacional de Colombia, Bogotá, 2013.
\bibitem{ReyesSuarez}\textbf{Reyes, A. and Suárez, H.}, \textit{A notion of compatibility for Armendariz and Baer properties over skew
PBW extensions}, Rev. Un. Mat. Argentina, 59 (1), 2018, 157-178.
\bibitem{reyessuarez}\textbf{Reyes, A. and Suárez, H.}, \textit{Sigma-PBW extensions of skew Armendariz rings}, Adv. Appl. Clifford Algebr., 27 (4), 2017, 3197-3224.
\bibitem{Rogalski}\textbf{Rogalski, D.}, \textit{Noncommutative projective geometry. In Noncommutative algebraic
geometry}, vol. 64 of Math. Sci. Res. Inst. Publ. Cambridge Univ. Press, New York, 2016, pp. 13-70.
\bibitem{Suarez2}\textbf{Suárez, H.}, \textit{N-Koszul algebras, Calabi-Yau algebras and skew PBW
extensions}, Ph.D. Thesis, Universidad Nacional de Colombia, Bogotá, 2017.
\bibitem{Suarez}\textbf{Suárez, H.}, \textit{Koszulity for graded skew PBW extensions}, Comm. Algebra,
45 (10), 2017, 4569-4580.
\bibitem{ReyesSuarez2}\textbf{Suárez, H. and Reyes, M. A.}, \textit{Koszulity for skew PBW extensions over fields}, JP Journal of Algebra,
Number Theory and Applications, 39, 2017, 181-203.
\bibitem{Venegas2}\textbf{Venegas, C.}, \textit{Automorphisms for skew PBW extensions and skew quantum polynomial rings},
Comm. Algebra, 42 (5), 2015, 1877-1897.
\bibitem{Weibel}{\textbf{Weibel, C.}}, \textit{An Introduction to Homological Algebra}, Cambridge
University Press, 1997.

\end{thebibliography}
\end{document}